\documentclass{amsart}
\usepackage{tikz}
\usepackage{hyperref}   
\usepackage{amssymb,xcolor}
\usepackage{stmaryrd} % \owedge

\usepackage{mathtools}  
\mathtoolsset{showonlyrefs}
\theoremstyle{plain}
\newtheorem{definition}{Definition}
\theoremstyle{plain}
\newtheorem{corollary}{Corollary}
\theoremstyle{plain}
\newtheorem{theorem}{Theorem}[section]
\theoremstyle{plain}
\newtheorem{lemma}[theorem]{Lemma}
\theoremstyle{remark}
\newtheorem{remark}[theorem]{Remark}
\theoremstyle{plain}
\newtheorem{proposition}[theorem]{Proposition}
\numberwithin{equation}{section}

\begin{document}

\title{Inverse mean curvature flow with a free boundary in hyperbolic space}

\author{Xiaoxiang Chai}
\address{Korea Institute for Advanced Study, Seoul 02455, South Korea}
\email{xxchai@kias.re.kr}

\begin{abstract}
  We study the inverse mean curvature flow with a free boundary supported on
  geodesic spheres in hyperbolic space. Starting from any convex hypersurface
  inside a geodesic ball with a free boundary, the flow converges to a totally
  geodesic disk in finite time. Using the convergence result, we show a
  Willmore type inequality.
\end{abstract}

{\maketitle}

\section{Introduction}

Let $B_0 $ be a geodesic ball of radius $\rho_0 > 0$ in hyperbolic space
$\mathbb{H}^{n + 1}$ and $M_0$ be a convex hypersurface sitting inside $B_0$
with free boundary supported on $\partial B_0$. We study the inverse mean
curvature flow $M_t$ starting from $M_0$. The flow $M_t$ is a family of free
boundary immersions
\[ F : [0, T^{\ast}) \times \mathbb{D} \to B_0 \]
of $n$ dimensional disks $\mathbb{D}$ into $B_0$ with a choice of normal
vector field $\nu$ and $M_t = F (t, \cdot)$. The immersions $F$ satisfies the
following evolution:
\begin{equation}
  \left\{\begin{array}{lll}
    \tfrac{\partial}{\partial t} F & = \tfrac{1}{H} \nu & \text{ in }
    \mathbb{D}\\
    \langle \nu, \eta \rangle & = 0 & \text{ on } \partial \mathbb{D},
  \end{array}\right. \label{imcf}
\end{equation}
where $F (t, \partial \mathbb{D}) \subset \partial B_0$ and $\eta$ is the
outward unit normal to $\partial B_0$ in $B_0$.

The inverse mean curvature flow is an expanding flow first studied by
{\cite{gerhardt-inverse-2011,urbas-expansion-1990}}. Starting from a
star-shaped mean convex hypersurface in Euclidean space, after rescaling the
hypersurfaces converges to a standard sphere.

The work of Huisken and Ilmanen {\cite{huisken-inverse-2001}} is closely
related to our study. Besides a weak theory, they utilized a Geroch
monotonicity of the Hawking mass {\cite{hawking-gravitational-1968}} to give a
proof of the Penrose inequality under the assumption of nonnegative scalar
curvature. Similar monotonicity were observed in other expanding flows and
geometric quantity preserving flows leading to rich results in convex
geometry, isoperimetric problems and general relativity. Hawking mass has a
natural generalization to the free boundary case. A similar monotonicity of
Hawking mass with boundary was observed \ under the free boundary inverse mean
curvature flow by {\cite{marquardt-weak-2017}}. Lambert and Scheuer
{\cite{lambert-inverse-2016}} developed a convergence result which says that a
free boundary closed convex hypersurface in the unit ball converges to a
totally geodesic disk under the inverse mean curvature flow. Later, they
{\cite{lambert-geometric-2017}} showed a Willmore type inequality in higher
dimensions. See the works \
{\cite{scheuer-alexandrov-fenchel-2018,wang-guan-li-2020}} for some further
interesting developments.

Unlike in Euclidean space, the inverse mean curvature flow starting from a
closed hypersurface in hyperbolic space does not always converge to a round
sphere after rescaling {\cite{gerhardt-inverse-2011}}. In the free boundary
case, Fraser-Schoen {\cite{fraser-uniqueness-2015}} showed the that minimal
immersions of disks in hyperbolic geodesic ball of dimension 3 can only be
totally geodesic. We expect that the free boundary inverse mean curvature flow
in hyperbolic geodesic ball converges to a totally geodesic disk and it is
hopeful to achieve a similar theory as in {\cite{lambert-inverse-2016}}. This
is the goal of our work. We have showed the following.

\begin{theorem}
  \label{convergence theorem}Let $F_0$ the embedding of a smooth and strictly
  convex free boundary hypersurface in a geodesic ball of radius $\rho_0$ with
  unit normal vector field $\nu$. Then there exists a finite time $T^{\ast}$,
  $0 < \alpha < 1$ and a unique solution such that
  \[ F \in C^{1 + \tfrac{\alpha}{2}, 2 + \alpha} ([0, T^{\ast}) \times
     \mathbb{D}) \cap C^{\infty} ((0, T^{\ast}) \times \mathbb{D}) \]
  of free boundary inverse mean curvature flow \eqref{imcf} with initial
  hypersurface $M_0$. The flow $F (t, \cdot)$ converges to a unique totally
  geodesic disk as $t \to T^{\ast}$.
\end{theorem}

Using this convergence result, we establish the Willmore type inequality.

\begin{theorem}
  \label{willmore}Let
  \[ \lambda = \omega_{n - 1} \int_0^{\rho_0} \sinh^{n - 1} s \mathrm{d} s \]
  be the volume of $n$-dimensional hyperbolic geodesic ball of radius $\rho_0$
  and $\Lambda = 2 \coth \rho_0 \lambda^{\tfrac{2 - n}{n}}$. Any weakly convex
  free boundary hypersurface $M$ in $B_0$ satisfies the Willmore type
  inequality
  \begin{equation}
    |M|^{\tfrac{2 - n}{n}} \int_M (H^2 - n^2) + \Lambda | \partial M|
    \geqslant - n^2 \lambda^{\tfrac{2}{n}} + \Lambda \omega_{n - 1} \sinh^{n -
    1} \rho_0, \label{willmore for weak convex}
  \end{equation}
  where $\omega_{n - 1}$ is the volume of standard $(n - 1)$-sphere. Equality
  occurs if and only if when $M$ is totally geodesic.
\end{theorem}

It is worth mentioning that Volkmann {\cite{volkmann-monotonicity-2016}}
established a Willmore type inequality for free boundary surfaces or more
generally two dimensional free boundary varifolds in the Euclidean unit ball.
Volkmann used a Simon type monotonicity argument. It seems an interesting
problem to generalize Volkmann's result to hyperbolic space.

The article is organized as follows:

In Section \ref{geometry of hypersurfaces}, we collect basic facts on the
geometry of hypersurfaces in hyperbolic space. In Section \ref{evolution and
boundary derivatives}, we calculate the evolutions and boundary derivatives of
various quantities under inverse mean curvature flow. In Section \ref{geometry
from convexity}, we study the geometry of a strictly convex free boundary
hypersurface from the point of view of convexity, and these results would be
used in Section \ref{moebius coordinates} to show that a convex free boundary
hypersurface is graphical over some hyperbolic subspace. In Section
\ref{important estimates}, we show important estimates vital to the proof of
Theorem \ref{convergence theorem}, establish the convergence theorem Theorem
\ref{convergence theorem}. In Section \ref{section willmore}, we prove the
Willmore inequality Theorem \ref{willmore} and characterize the equality case.

\

\text{{\bfseries{Acknowledgments}}} Research of the author is supported by
KIAS Grants under the research code MG074402. The author would also like to
thank Yizi Wang of Chinese University of Hong Kong for some useful discussions
and Julian Scheuer of Cardiff University for pointing out some inaccuracies in
an earlier version of the paper.

\section{geometry of hypersurfaces}\label{geometry of hypersurfaces}

We make heavy use the hyperboloidal model of the hyperbolic space. The book
{\cite{benedetti-lectures-1992}} is a good reference for basics of hyperbolic
spaces. The hyperbolic space $\mathbb{H}^{n + 1}$ can be realized as the upper
sheet of the two-sheeted hyperboloid in Minkowski space $\mathbb{R}^{n + 1,
1}$, that is
\[ \mathbb{H}^{n + 1} = \{z \in \mathbb{R}^{n + 1, 1} : \langle z, z \rangle =
   - 1, z^0 > 0\}, \]
where
\[ \langle z, w \rangle = - z^0 w^0 + z^1 w^1 + \cdots + z^{n + 1} w^{n + 1} .
\]
The single-sheeted hyperboloid in Minkowski space known as the de Sitter
sphere is a timelike hypersurface and given by
\[ \mathrm{d} \mathbb{S}^{n, 1} = \{z \in \mathbb{R}^{n + 1, 1} : \langle z, z
   \rangle = 1\} . \]
The Poincar{\'e} ball model of hyperbolic space is good for visualization of
some concepts. A point $z$ of the hyperboloid model is send to the
hyperboloidal coordinates via the stereographic projection $\pi$ and
\[ x^i = \tfrac{z^i}{1 + z^0}, i \geqslant 1 \]
of the Poincar{\'e} ball model. The metric $b$ of the ball model is then given
by
\[ b = \tfrac{4}{(1 - |x|^2)^2} | \mathrm{d} x|^2, |x| < 1. \]
On the other hand, a point $x$ in the Poincar{\'e} ball model is sent to
\begin{equation}
  z^0 = \tfrac{1 + |x|^2}{1 - |x|^2}, z^i = \tfrac{2 x^i}{1 - |x|^2}, i
  \geqslant 1 \label{x to z}
\end{equation}
of the hyperboloid. This is the inverse of the stereographic projection.

The distance $\rho (z, w)$ between two points $z, w \in \mathbb{H}^{n + 1}$ in
hyperboloid model is
\begin{equation}
  \cosh \rho (z, w) = - \langle z, w \rangle . \label{distance}
\end{equation}
In particular, the distance $\rho (z)$ from $z$ to $e_0 = (1, 0)$ satisfies
\begin{equation}
  \cosh \rho (z) = z^0 \label{distance to origin} .
\end{equation}
In the ball model, the distance $\rho_x$ from $x$ to $o = (0, \ldots, 0)$ is
\begin{equation}
  \rho_x = \int_0^{|x|} \tfrac{2}{1 - |s|^2} \mathrm{d} s = \log \left(
  \tfrac{1 + |x|}{1 - |x|} \right), \sinh \rho_x = \tfrac{2 |x|}{1 - |x|^2} .
  \label{distance of ball model}
\end{equation}
A geodesic sphere $\partial B_0$ centered at $(1, 0)$ is given by
\[ \{z \in \mathbb{H}^{n + 1} : z^0 = \cosh \rho_0 \} . \]
By the distance formula \eqref{distance to origin}, the outward normal to
$\partial B_0$ is
\[ \nabla \rho = \tfrac{1}{\sinh \rho} \nabla z^0 = \tfrac{1}{\sinh \rho} (-
   e_0 + z^0 z) . \]
A typical example of a free boundary convex hypersurface inside the ball $B_0$
is the spherical cap which is part of a geodesic sphere.

\subsection{Differential geometry of hypersurfaces}

An immersion into the hyperbolic space which is given by a map
\[ z : \mathbb{D} \to \mathbb{R}^{n + 1, 1} \]
from an $n$-dimensional disk $\mathbb{D}$ in $\mathbb{H}^{n + 1}$ is a
codimension two submanifold with boundary in $\mathbb{R}^{n + 1, 1}$. We
denote this hypersurface by $M$. Let $\xi^i$ be the coordinates on
$\mathbb{D}$, then $z_i : = \tfrac{\partial z}{\partial \xi^i}$ is a tangent
vector at $z (\xi)$, satisfying the relation $\langle z, z_i \rangle = 0$. We
reserve the superscript notation $z^i$ and the alike to denote coordinate
components and subscript to denote taking derivatives. The induced metric on
$M$ is then
\[ g_{i j} = \langle z_i, z_j \rangle . \]
Let $\nu$ be the normal vector of $M$ in $\mathbb{H}^{n + 1}$ at $z$, the
second fundamental form $h_{i j}$ of $M$ in $\mathbb{H}^{n + 1}$ is
\[ h_{i j} = - \langle z_{i j}, \nu \rangle . \]
We have the Weingarten relation
\begin{equation}
  z_{i j} = g_{i j} z - h_{i j} \nu . \label{hessian z}
\end{equation}
Taking the trace of the above, and written in components,
\begin{equation}
  \Delta z^{\alpha} = n z^{\alpha} - H \nu^{\alpha}, \alpha = 0, 1, \ldots, n
  + 1. \label{laplacian z}
\end{equation}
Replacing the Hessian $\nabla_i \nabla_j$ in \eqref{hessian z} by the Hessian
of the hyperbolic space, we use the letter $D$ to denote connection on
$\mathbb{H}^{n + 1}$, instead we get
\[ D_{\alpha} D_{\beta} z = b_{\alpha \beta} z. \]
The functions $z^0$ and $z^i$, $i = 1, \ldots, n + 1$ are called static
potentials in general relativity literature {\cite{chrusciel-mass-2003}}. The
components $\nu^{\alpha}$ of the normal $\nu$ are also
\[ \nu^{\alpha} = \langle D z^{\alpha}, \nu \rangle . \]
The vector $\nu$ is spacelike in $\mathbb{R}^{n + 1, 1}$, hence an element of
the de Sitter sphere $\mathrm{d} \mathbb{S}^{n, 1}$. We call the map $\nu : M
\to \mathrm{d} \mathbb{S}^{n, 1}$ the \text{{\itshape{Gauss map}}}, and to
emphasize the dependence on $z$, we use the notation $\tilde{z} := \nu$ as
well.

\subsection{Christoffel symbols in the Poincar{\'e} ball model}

Let $u = \tfrac{2}{1 - |x|^2}$. Then
\[ D_i \log u = \frac{2 x^i}{1 - |x|^2} = z^i . \]
The metric $b$ is conformal to Euclidean metric, we use $\partial_i$ to denote
the $i$-th unit vector in Euclidean space, so the Christoffel symbol defined
by $D_j \partial_i = \Gamma_{i j}^k \partial_k$ is given in the following
\begin{align}
& \Gamma_{i j}^k \\
= & \delta_i^k D_j \log u + \delta_j^k D_i \log u - \delta_{i j} D^k \log u
\\
= & \delta_i^k z^j + \delta_j^k z^i - \delta_{i j} z_k .
\end{align}
So
\[ D_j \partial_i = z^j \partial_i + z^i \partial_j - \delta_{i j} \hat{z}, \]
where we use $\hat{z} = (z^1, \ldots, z^{n + 1}) \in \mathbb{B}^{n + 1}$.

The formula
\[ D_j (x^l \partial_l) = \partial_j + x^l D_j \partial_l \]
is useful for later use.

\subsection{Useful computations}

The vector $\partial_i$ in the Poincar{\'e} ball model is (by stereoprojection
map $\pi$) under the hyperboloidal model is
\[ \pi_{\ast} (\partial_i) = \frac{4 x^i}{(1 - |x|^2)^2} e_0 + \frac{2}{1 -
   |x|^2} e_i + \frac{4 x^i x^j}{(1 - |x|^2)^2} e_j, \]
Note that $1 - |x|^2 = \tfrac{2}{1 + z^0}$, so we have that
\[ \pi_{\ast} (\partial_i) = (1 + z^0) z^i e_0 + (1 + z^0) e_i + z^i z^j e_j .
\]
From the above,
\begin{equation}
  \pi_{\ast} (x) = ((z^0)^2 - 1) e_0 + z^0 z^i e_i . \label{conformal killing
  vector x}
\end{equation}
And
\begin{equation}
  \langle \pi_{\ast} (x), \nu \rangle = - ((z^0)^2 - 1) \nu^i + z^0 z^i \nu^i
  = \nu^0, \label{normal component of x}
\end{equation}
where we have used $\langle z, \nu \rangle = 0$.

\section{Evolution equations and boundary derivatives}\label{evolution and
boundary derivatives}

In this section, we derive evolution equations and boundary derivatives of
various quantities.

\begin{lemma}
  Let $X$ be a tangent vector field on $\partial M$, then
  \begin{equation}
    A (X, \eta) = 0. \label{second fundamental form cross term}
  \end{equation}
\end{lemma}

\begin{proof}
  We differentiate the relation $\langle \eta, \nu \rangle = 0$ along $X$ on
  $\partial M$ and get
  \[ \langle D_X \eta, \nu \rangle + \langle D_X \nu, \eta \rangle = 0. \]
  So
  \[ A (X, \eta) = \langle D_X \nu, \eta \rangle = - \langle D_X \eta, \nu
     \rangle . \]
  Recall that the vector $\eta$ is also outward normal of the geodesic sphere
  $\partial B_0$, $X$ and $\nu$ are tangent to $\partial B_0$, so $\langle D_X
  \eta, \nu \rangle = 0$. Hence $A (X, \eta) = 0$.
\end{proof}

\begin{lemma}
  \label{lemma normal derivative h ij}Let $X$ and $Y$ be tangent vector fields
  on $\partial M$,
  \begin{equation}
    (\nabla_{\eta} A) (X, Y) = \coth \rho_0 [A (\eta, \eta) \langle X, Y
    \rangle - A (X, Y)] . \label{normal derivative h ij}
  \end{equation}
\end{lemma}

\begin{proof}
  We differentiate \eqref{second fundamental form cross term} along $Y$,
  \[ \nabla_Y (A (X, \eta)) = 0. \]
  By product rule,
  \[ (\nabla_Y A) (X, \eta) = - A (\nabla_Y X, \eta) - A (X, \nabla_Y \eta) .
  \]
  In $T M$, the vector field is decomposed into components normal to $\eta$
  and parallel to $\eta$, but because of \eqref{second fundamental form cross
  term},
  \[ A (\nabla_Y X, \eta) = A (\eta, \eta) \langle \nabla_Y X, \eta \rangle =
     - A (\eta, \eta) \langle \nabla_Y \eta, X \rangle, \]
  and that $\langle \nabla_Y \eta, X \rangle = \coth \rho_0 \langle X, Y
  \rangle$, so
  \[ - A (\nabla_Y X, \eta) = \coth \rho_0 A (\eta, \eta) \langle X, Y \rangle
     . \]
  The vector field $\eta$ is unit normal to $\partial M$ in $M$, so $\langle
  \nabla_Y \eta, \eta \rangle = 0$ and $\nabla_Y \eta$ coincides with the
  second fundamental form of $\partial B_0$ in $B_0$, so $\nabla_Y \eta =
  \coth \rho_0 Y$. So
  \[ - A (X, \nabla_Y \eta) = - \coth \rho_0 A (X, Y) . \]
  And by Codazzi equation,
  \[ (\nabla_Y A) (\eta, X) = (\nabla_{\eta} A) (X, Y) . \]
  Collecting all the above leads to \eqref{normal derivative h ij}.
\end{proof}

\begin{lemma}
  \label{lemma normal derivative h 11}We have
  \begin{equation}
    (\nabla_{\eta} A) (\eta, \eta) = - n \coth \rho_0 A (\eta, \eta) .
    \label{normal derivative h 11}
  \end{equation}
\end{lemma}

\begin{proof}
  We decompose $H$, so
  \[ - \coth \rho_0 H = \nabla_{\eta} H = [(g^{i j} - \eta^i \eta^j) + \eta^i
     \eta^j] \nabla_{\eta} h_{i j} . \]
  Note that $(g^{i j} - \eta^i \eta^j) \nabla_{\eta} h_{i j}$ is just trace of
  the tensor $\nabla_{\eta} h_{i j}$ restricted to tangent space of $\partial
  M$, from \eqref{normal derivative h ij},
\begin{align}
& (g^{i j} - \eta^i \eta^j) \nabla_{\eta} h_{i j} \\
= & \coth \rho_0 [A (\eta, \eta) (n - 1)
-\ensuremath{\operatorname{tr}}_{\partial M} A] \\
= & \coth \rho_0 [n A (\eta, \eta) - H],
\end{align}
  leads to \eqref{normal derivative h 11}.
\end{proof}

\subsection{Curvatures}

We are interested in the evolution of the second fundamental form and the mean
curvature. They are given by the following.

\begin{lemma}
  \label{evolution H normal derivative H}The Weingarten tensor $h_i^j := g^{j
  k} h_{i k}$ evolves by
  \begin{equation}
    \partial_t h^j_i = - \nabla_i {\nabla^j}  \tfrac{1}{H} - (- \delta_i^j +
    h_i^k h_k^j) \tfrac{1}{H} . \label{evolution weingarten}
  \end{equation}
  The quantity $\log H$ evolves by the equation
  \begin{equation}
    (\partial_t - \tfrac{1}{H^2} \Delta) \log H = - \tfrac{| \nabla H|^2}{H^4}
    - (- n + |A|^2) \tfrac{1}{H^2}, \label{evolution log H}
  \end{equation}
  and on the boundary $\partial M$,
  \begin{equation}
    \nabla_{\eta} \log H = - \coth \rho_0 . \label{normal derivative H}
  \end{equation}
\end{lemma}

\begin{proof}
  For the proof \eqref{evolution weingarten}, see for example
  {\cite{andrews-contraction-1994}}. Taking the trace of \eqref{evolution
  weingarten},
  \[ \partial_t H = - \Delta \tfrac{1}{H} - (- n + |A|^2) \tfrac{1}{H} . \]
  This implies
  \[ (\partial_t - \tfrac{1}{H^2} \Delta) H = - \tfrac{2}{H^3} | \nabla H|^2 -
     (- n + |A|^2) \tfrac{1}{H} . \]
  and
  \[ (\partial_t - \tfrac{1}{H^2} \Delta) \tfrac{1}{H} = (- n + |A|^2)
     \tfrac{1}{H^3} . \]
  We obtain the evolution for $\log H$, and $\nabla_{\eta} \log H = - \coth
  \rho_0$ follows from {\cite{stahl-convergence-1996}}.
\end{proof}

\begin{lemma}
  Let $\tilde{h}_{i j}$ be the inverse matrix of $h_{i j}$, then the evolution
  of $\tilde{H} := g_{i j} \tilde{h}^{i j}$ is given by
\begin{align}
& \partial_t \log \tilde{H} - \tfrac{1}{H^2} \Delta \log \tilde{H}
\\
= & \tfrac{2 n}{\tilde{H} H} - \tfrac{n}{H^2} - \tfrac{|A|^2}{H^2} +
\tfrac{1}{H^2 \tilde{H}^2} | \nabla \tilde{H} |^2 \\
& + \tfrac{2}{H^2 \tilde{H}} \left( \tfrac{1}{H} g_{i j} \tilde{h}^{i k}
\tilde{h}^{j l} \nabla_k H \nabla_l H - g_{i j} \tilde{h}^{i k}
\tilde{h}^{j p} \tilde{h}^{l q} \nabla h_{p q} \cdot \nabla h_{k l}
\right) . \label{evolution harmonic mean curvature}
\end{align}
\end{lemma}

\begin{proof}
  We write in short $\phi = \tfrac{1}{H}$ for some convenience. First,
  $\partial_t g_{i j} = 2 \phi h_{i j}$ and
  \[ \partial_t h_i^j = - \nabla_i \nabla^j \phi - \phi (h_i^k h_k^j -
     \delta_i^j) . \]
  So
\begin{align}
& \partial_t h_{i j} \\
= & \partial_t (g_{i k} h_j^k) \\
= & h_j^k \partial_t g_{i k} + g_{i k} \partial_t h_j^k \\
= & 2 \phi h_j^k h_{i k} - \nabla_i \nabla_j \phi - \phi (h_i^k h_{k j} -
g_{i j}) \\
= & - \nabla_i \nabla_j \phi + \phi g^{k s} h_{i s} h_{j k} + \phi g_{i j}
.
\end{align}
  Now
\begin{align}
& \partial_t \tilde{H} \\
= & \partial_t (g_{i j} \tilde{h}^{i j}) \\
= & \tilde{h}_{i j} \partial_t g_{i j} + g_{i j} \partial_t \tilde{h}^{i
j} \\
= & 2 \phi h_{i j} \tilde{h}^{i j} - g_{i j} \tilde{h}^{i k} \tilde{h}^{j
l} \partial_t h_{k l} \\
= & 2 n \phi - g_{i j} \tilde{h}^{i k} \tilde{h}^{j l} (- \nabla_k
\nabla_l \phi + \phi g^{p s} h_{k s} h_{l p} + \phi g_{k l}) \\
= & n \phi + g_{i j} \tilde{h}^{i k} \tilde{h}^{j l} \nabla_k \nabla_l
\phi - \phi g_{i j} \tilde{h}^{i k} \tilde{h}^{j l} g_{k l} \\
= & \tfrac{n}{H} - g_{i j} \tilde{h}^{i k} \tilde{h}^{j l} \tfrac{1}{H^2}
\nabla_k \nabla_l H + 2 g_{i j} \tilde{h}^{i k} \tilde{h}^{j l}
\tfrac{1}{H^3} \nabla_k H \nabla_l H - g_{i j} \tilde{h}^{i k}
\tilde{h}^{j l} g_{k l} \tfrac{1}{H} .
\end{align}
  We compute now $\Delta \tilde{H}$ and we have
\begin{align}
& \Delta \tilde{H} \\
= & g_{i j} \Delta \tilde{h}^{i j} \\
= & - g_{i j} \nabla (\tilde{h}^{i k} \tilde{h}^{j l} \nabla h_{k l})
\\
= & - g_{i j} \tilde{h}^{i k} \tilde{h}^{j l} \Delta h_{k l} - g_{i j}
\tilde{h}^{i k} \nabla \tilde{h}^{j l} \cdot \nabla h_{k l} - g_{i j}
\tilde{h}^{j l} \nabla \tilde{h}^{i k} \cdot \nabla h_{k l} .
\end{align}
  We give a quick derivation of the Simons identity:
\begin{align}
& \Delta h_{i j} \\
= & \nabla_k \nabla_k h_{i j} \\
= & \nabla_k \nabla_i h_{k j} \\
= & \nabla_i \nabla_k h_{k j} - R_{k i k}^{\phantom{i j k} l} h_{l j} -
R_{k i j}^{\phantom{i j k} l} h_{k l} \\
= & \nabla_i \nabla_j h_{k k} - (\bar{R}_{k i k}^{\phantom{i j k} l} +
h_{i k} h^l_k - h_{k k} h_i^l) h_{l j} \\
& - (\bar{R}_{k i j}^{\phantom{i j k} l} + h_{i j} h_k^l - h_{k j} h_i^l)
h_{k l} \\
= & \nabla_i \nabla_j H - (\bar{g}_{k k} \bar{g}_i^l - \bar{g}_k^l
\bar{g}_{i k} + h_{i k} h^l_k - h_{k k} h_i^l) h_{j l} \\
& - (\bar{g}_{k j} \bar{g}_i^l - \bar{g}_k^l \bar{g}_{i j} + h_{i j}
h_k^l - h_{k j} h_i^l) h_{k l} \\
= & \nabla_i \nabla_j H - n h_{i j} + H h_i^l h_{j l} + H g_{i j} - h_{i
j} |A|^2 \label{simons} .
\end{align}
  So
\begin{align}
& \Delta \tilde{H} \\
= & - g_{i j} \tilde{h}^{i k} \tilde{h}^{j l} (\nabla_k \nabla_l H - n
h_{k l} + H h_k^s h_{l s} + H g_{k l} - h_{k l} |A|^2) \\
& - g_{i j} \tilde{h}^{i k} \nabla \tilde{h}^{j l} \cdot \nabla h_{k l} -
g_{i j} \tilde{h}^{j l} \nabla \tilde{h}^{i k} \cdot \nabla h_{k l}
\\
= & - g_{i j} \tilde{h}^{i k} \tilde{h}^{j l} \nabla_k \nabla_l H + n
\tilde{H} - n H - H g_{i j} g_{k l} \tilde{h}^{i k} \tilde{h}^{j l} +
\tilde{H} |A|^2 \\
& - g_{i j} \tilde{h}^{i k} \nabla \tilde{h}^{j l} \cdot \nabla h_{k l} -
g_{i j} \tilde{h}^{j l} \nabla \tilde{h}^{i k} \cdot \nabla h_{k l} .
\end{align}
  Now
\begin{align}
& (\partial_t - \tfrac{1}{H^2} \Delta) \log \tilde{H} \\
= & \tfrac{1}{\tilde{H}} \partial_t \tilde{H} - \tfrac{1}{H^2 \tilde{H}}
\Delta \tilde{H} + \tfrac{1}{H^2 \tilde{H}^2} | \nabla \tilde{H} |^2
\\
= & \tfrac{2 n}{\tilde{H} H} + 2 g_{i j} \tilde{h}^{i k} \tilde{h}^{j l}
\tfrac{1}{H^3 \tilde{H}} \nabla_k H \nabla_l H \\
& - \tfrac{n}{H^2} - \tfrac{|A|^2}{H^2} + \tfrac{1}{\tilde{H} H^2} g_{i
j} (\tilde{h}^{i k} \nabla \tilde{h}^{j l} + \tilde{h}^{j l} \nabla
\tilde{h}^{i k}) \cdot \nabla h_{k l} \\
& + \tfrac{1}{H^2 \tilde{H}^2} | \nabla \tilde{H} |^2 .
\end{align}
  Considering that $\tilde{h}^{i j}$ is the inverse matrix of $h_{i j}$,
  $\nabla \tilde{h}^{i j} = - h^{i k} h^{j l} \nabla h_{k l}$ and symmetry of
  indices leads to \eqref{evolution harmonic mean curvature}.
\end{proof}

\begin{lemma}
  If $M_t$ is convex, then boundary derivatives $\nabla_{\eta} \tilde{H}$
  satisfies the estimate
  \begin{equation}
    \nabla_{\eta} \tilde{H} \leqslant n \coth \rho_0 \tilde{H} . \label{normal
    derivative tilde H}
  \end{equation}
\end{lemma}

\begin{proof}
  We fix coordinates on $M_t$ such that $\partial_1 = \eta$, so \eqref{normal
  derivative h 11} and \eqref{normal derivative h ij} implies that
\begin{align}
& \nabla_{\eta} \tilde{H} \\
= & - \eta^k g_{i j} \tilde{h}^{i l} \tilde{h}^{j s} \nabla_k h_{l s}
\\
= & - g_{i j} \tilde{h}^{1 i} \tilde{h}^{1 j} \nabla_1 h_{1 1} - \sum_{l,
s \neq 1} g_{i j} \tilde{h}^{i l} \tilde{h}^{j s} \nabla_1 h_{l s}
\\
= & n \coth \rho_0 g_{i j} \tilde{h}^{1 i} \tilde{h}^{1 j} h_{1 1} -
\sum_{l, s \neq 1} \coth \rho_0 g_{i j} \tilde{h}^{i l} \tilde{h}^{j s}
(h_{11} g_{l s} - h_{l s}) .
\end{align}
  Since the cross term \eqref{second fundamental form cross term}, so
  $\tilde{h}^{1 i} = 0$ for $i \neq 1$ and we can assume that $g_{i j}$ for
  $i, j \geqslant 2$ is unit matrix, then $h_{i j}$ is diagonal matrix with
  diagonal entries $(\kappa_1, \ldots, \kappa_n) .$ So $\tilde{h}^{i j}$ is
  diagonal with diagonal entries $(\kappa_1^{- 1}, \ldots, \kappa_n^{- 1})$.
  We see then
\begin{align}
\nabla_{\eta} \tilde{H} & = \coth \rho_0 (\tfrac{n}{\kappa_1} - \kappa_1
\sum_{i \neq 1} \kappa_i^{- 2} + \sum_{i \neq 1} \kappa_i^{- 1})
\\
& \leqslant \coth \rho_0 (\tfrac{n - 1}{\kappa_1} + \tilde{H}) .
\end{align}
  Since $\tilde{H} = \sum_i \kappa_i^{- 1}$ obviously, so $\nabla_{\eta}
  \tilde{H} \leqslant n \coth \rho_0 \tilde{H}$.
\end{proof}

\subsection{Hyperboloidal coordinates}

The functions $z^0$ and $z^i$ are used in later sections to construct
auxiliary functions.

\begin{lemma}
  The functions $z^0$ and $z^i$ evolves by
  \begin{equation}
    \partial_t z^i - \tfrac{1}{H^2} \Delta z^i = - \tfrac{n}{H^2} z^i +
    \tfrac{2}{H} \nu^i . \label{evolution z}
  \end{equation}
\end{lemma}

\begin{proof}
  First, by the equation of the inverse mean curvature flow \eqref{imcf},
  \[ \partial_t z^i = \tfrac{1}{H} \nu^i . \]
  Combining with \eqref{laplacian z}, we obtain \eqref{evolution z}.
\end{proof}

We are also interested in the evolution of $f (z)$.

\begin{lemma}
  Evolution of $f (z)$ is given by
  \begin{equation}
    (\partial_t - \tfrac{1}{H^2} \Delta) f = - \tfrac{n}{H^2} \tfrac{\partial
    f}{\partial z^{\alpha}} z^{\alpha} + \tfrac{2}{H} \tfrac{\partial
    f}{\partial z^{\alpha}} \nu^{\alpha} - \tfrac{\partial^2 f}{H^2 \partial
    z^{\alpha} \partial z^{\beta}} \langle \nabla z^{\alpha}, \nabla z^{\beta}
    \rangle . \label{evolution f z}
  \end{equation}
  In particular, when $f = \log z^{\alpha}$
  \[ (\partial_t - \tfrac{1}{H^2} \Delta) \log z^{\alpha} = - \tfrac{n}{H^2} +
     \tfrac{2}{H z^{\alpha}} \nu^{\alpha} + \tfrac{1}{H^2 (z^{\alpha})^2} |
     \nabla z^{\alpha} |^2 . \]
\end{lemma}

\begin{proof}
  First, by \eqref{evolution z},
  \[ \partial_t f (z) = \tfrac{\partial f}{\partial z^{\alpha}} \partial_t
     z^{\alpha} = \tfrac{\partial f}{\partial z^{\alpha}} (\tfrac{1}{H^2}
     \Delta z^{\alpha} - \tfrac{n}{H^2} z^{\alpha} + \tfrac{2}{H}
     \nu^{\alpha}) . \]
  By Leibniz rule,
\begin{align}
& \Delta f (z) \\
= & \nabla^i (\tfrac{\partial f}{\partial z^{\alpha}} \nabla_i z^{\alpha})
\\
= & \tfrac{\partial f}{\partial z^{\alpha}} \Delta z^{\alpha} +
\tfrac{\partial^2 f}{\partial z^{\alpha} \partial z^{\beta}} \langle
\nabla z^{\alpha}, \nabla z^{\beta} \rangle .
\end{align}
  Subtraction the above from $\partial_t f (z)$ gives \eqref{evolution f z}.
\end{proof}

\begin{lemma}
  The boundary derivatives of $\nabla_{\eta} z^i$ and $\nabla_{\eta} z^i$ are
  given by
  \begin{equation}
    \nabla_{\eta} \log z^0 = \tanh \rho_0, \nabla_{\eta} z^i = \coth \rho_0
    z^i . \label{normal derivative z}
  \end{equation}
\end{lemma}

\begin{proof}
  We use the hyperboloidal coordinates. The normal $N$ to the geodesic sphere
  $\partial B (\rho_0)$ is just $\nabla \rho$, where $\rho$ is the distance
  function to $e_0$. By \eqref{distance to origin},
  \[ \eta = \nabla \rho = \tfrac{1}{\sinh \rho} \nabla z^0 = \tfrac{1}{\sinh
     \rho} (- e_0 + z^0 z) . \]
  It readily leads to
  \[ \nabla_{\eta} \log z^0 = \tfrac{\nabla_{\eta} z^0}{z^0} = \tfrac{(z^0)^2
     - 1}{\sinh \rho \cosh \rho} = \tanh \rho, \]
  and
  \[ \nabla_{\eta} z^i = \tfrac{1}{\sinh \rho} z^0 z^i = \coth \rho  z^i . \]
  We can also calculate the normal derivatives $\nabla_{\eta} z$ by using the
  Poincar{\'e} ball model. The calculation is slightly longer.
\end{proof}

\section{Geometry from convexity}\label{geometry from convexity}

First, we show that $\partial M$ is a convex hypersurface in $\partial B_0$.

\begin{lemma}
  Let $M$ be a convex free boundary hypersurface in $B_0$, then $\partial M$
  is a closed convex hypersurface in $\partial B_0$.
\end{lemma}

\begin{proof}
  The convexity of $\partial M$ in $\partial B_0$ readily follows from the
  free boundary condition and that the normal of $M$ in $B_0$ is also the
  normal of $\partial M$ in $\partial B_0$.
\end{proof}

The following lemma states that $\langle z, \tilde{y} \rangle$ is geometric
and does not depend on the choice of embedding of $\mathbb{H}^{n + 1}$ into
$\mathbb{R}^{n + 1, 1}$. The quantity $\langle z, \tilde{y} \rangle$ is from
the work {\cite{gerhardt-minkowski-2006}} and closely related to convexity.

\begin{lemma}
  \label{distance to a hypersurface}Let $y$ a point in $\mathbb{H}^{n + 1}$,
  $\tilde{y}$ be a vector in the tangent space $T_y \mathbb{H}^{n + 1}$ and
  $P$ be the subspace $\mathbb{H}^{n + 1} \cap \{x : \langle x, \tilde{y}
  \rangle = 0\}$. Let $d$ be the distance from a point $z' \in \mathbb{H}^{n +
  1}$ to $P$. By assigning $d$ to be positive if $z'$ lies in the same side
  with which $\tilde{y}$ points into and negative if $z'$ lies in the opposite
  side, we can view $d$ as a signed distance, and moreover
  \[ \sinh d = \langle z', \tilde{y} \rangle . \]
\end{lemma}

\begin{proof}
  The vectors $y$ and $\tilde{y}$ are orthogonal unit vectors in
  $\mathbb{R}^{n + 1, 1}$, we extend $\{y, \tilde{y} \}$ to an orthonormal
  basis of $\mathbb{R}^{n + 1, 1}$. Any point $z' \in \mathbb{H}^{n + 1}$ can
  then be written as
  \[ z' = y \cosh \rho' + \xi' \sinh \rho', \rho' \geqslant 0, \]
  where $\xi'$ is a unit vector orthogonal to $y$ and $\rho'$ is the distance
  in $\mathbb{H}^{n + 1}$ from $y$ to $z'$. Any point $z$ in $P$ can be
  written as $z = y \cosh \rho + \xi \sinh \rho$ where $\rho$ is the distance
  in $\mathbb{H}^{n + 1}$ from $y$ to $z$, and $\xi$ is a unit vector in
  $\mathbb{R}^{n + 1, 1}$ with $\langle \xi, \tilde{y} \rangle = \langle \xi,
  y \rangle = 0$. We see from \eqref{distance} we just have to find out the
  maximum of $\langle z, z' \rangle$. Since
\begin{align}
& - \cosh \ensuremath{\operatorname{dist}} (z, z') = \langle z, z'
\rangle \\
= & - \cosh \rho \cosh \rho' + \langle \xi, \xi' \rangle \sinh \rho' \sinh
\rho \\
\leqslant & - \cosh \rho \cosh \rho' + \sinh \rho' \sinh \rho \sqrt{1 -
\langle \xi', \tilde{y} \rangle^2} = F (\rho) .
\end{align}
  So $F (\rho)$ achieve its maximum when
  \[ - \sinh \rho \cosh \rho' + \sinh \rho' \cosh \rho \sqrt{1 - \langle \xi',
     \tilde{y} \rangle^2} = 0. \]
  This gives
  \[ \tanh^2 \rho = \tanh^2 \rho' (1 - \langle \xi', \tilde{y} \rangle^2) . \]
  Now we calculate the maximum of $F (\rho)$ using the condition above,
\begin{align}
& F (\rho) \\
= & - \cosh \rho' (\cosh \rho - \sinh \rho \tanh \rho' \sqrt{1 - \langle
\xi', \tilde{y} \rangle^2}) \\
= & - \cosh \rho' (\cosh \rho - \sinh \rho \tanh \rho) \\
= & - \tfrac{\cosh \rho'}{\cosh \rho} \\
= & - \cosh \rho' \sqrt{1 - \tanh^2 \rho} \\
= & - \cosh \rho' \sqrt{1 - \tanh^2 \rho' + \tanh^2 \rho' \langle \xi',
\tilde{y} \rangle^2} \\
= & - \sqrt{\cosh^2 \rho' - \sinh^2 \rho' + \sinh^2 \rho' \langle \xi',
\tilde{y} \rangle^2} \\
= & - \sqrt{1 + \sinh^2 \rho' \langle \xi', \tilde{y} \rangle} \\
= & - \cosh d, \\
= & - \sqrt{1 + \langle z', \tilde{y} \rangle^2} .
\end{align}
  where $d$ is the distance from $z'$ to the hyperbolic subspace determined by
  $\tilde{y}$. Note that $\langle z', \tilde{y} \rangle = \langle \tilde{y},
  \xi' \rangle \sinh \rho'$.
\end{proof}

The corollary below immediately follows from \eqref{hessian z} and convexity
of $M$.

\begin{corollary}
  \label{local separation}For each point $z$ of $M$, there exists a small
  neighborhood $U_z$ in $\mathbb{H}^{n + 1}$ of $z$ such that all points in
  $(M \cap U_z) \backslash \{z\}$ lie strictly on the opposite side of
  $\tilde{z}$ in the hyperbolic subspace orthogonal to $\tilde{z}$.
\end{corollary}

Note that it allows more choices of $\tilde{z}$ if $z$ is a boundary point of
$M$. We show in the following that through a point in the free boundary
$\partial M$ there is a 2-subspace $P$ such that points in a tiny neighborhood
of $z$ in $M \cap P$ lies on one side of a geodesic in $P$.

\begin{lemma}
  \label{2D local separation}Let $z \in \partial M$, $U_z$ be a neighborhood
  as in Corollary \ref{local separation} and $P$ be a 2-subspace such that $P
  \cap M$ contains $z$ and at least one more point from $U_z \cap M$. Then
  there exists a geodesic line $L$ such that $(P \cap M \cap U_z) \backslash
  \{z\}$ lie on one side of $L$.
\end{lemma}

\begin{proof}
  Let $\tilde{z}$ be a vector as in Corollary \ref{local separation}, then for
  any point $w \in U \cap M$, we have
  \[ \langle w, \tilde{z} \rangle < 0. \]
  Let $e_i$, $i = 1, 2$ be orthonormal tangent vectors at $P$. We extend $e_0
  := z$, $e_1$, and $e_2$ to a set of orthonormal basis of $\mathbb{R}^{n + 1,
  1}$. We write $\tilde{z} = \sum_{i = 0}^{n + 1} \tilde{z}^i e_i$ and any
  point $w = \sum_{i = 0}^2 w^i e_i$ in $P$. We define $v = \sum_{i = 0}^2
  \tilde{z}^i e_i$. Obviously $v$ is a tangent vector in $P$, and
  \[ \langle w, v \rangle = \langle \widetilde{z,} w \rangle < 0. \]
  Geometrically, $v$ is just the projection of $\tilde{z}$ to $P$ and it
  determines a geodesic line $L$ such that $(U \cap P \cap M) \backslash
  \{z\}$ lies on one side of $L$.
\end{proof}

\begin{lemma}
  \label{convexity via position normal relation}Let $\Omega'$ be the region
  bounded by $\partial B_0$ and $M$ such that $\nu$ points outward of
  $\Omega'$, and $\Sigma = \partial \Omega'$, and $y$ be a point in $\Sigma$,
  then
  \begin{equation}
    \langle z, \tilde{y} \rangle \leqslant 0 \label{position normal relation}
  \end{equation}
  for all $z \in \Sigma$ and equality occurring only when $z = y$.
\end{lemma}

\begin{proof}
  Note that Corollary \ref{local separation} and Lemma \ref{2D local
  separation} are valid at $\Sigma \backslash M$ as well. First, we assume
  that $y$ is a smooth point of $\Sigma$. Note that by Corollary \ref{local
  separation}, for all points $z \in \Sigma$ near $y$, $\langle z, \tilde{y}
  \rangle < 0$. Assume that for some $z \in \Sigma$ that $\langle z, \tilde{y}
  \rangle > 0$. Then by continuity, there is at least another point not $y$
  itself in $M$ which we still $w'$ satisfies $\langle w', \tilde{y} \rangle =
  0$. By Lemma \ref{distance to a hypersurface}, $w'$ lies in the subspace
  normal to $\tilde{y}$. There is a geodesic in $\mathbb{H}^{n + 1}$ from $y$
  to $w'$, there is a unique hyperbolic 2-subspace $P$ linearly spanned by
  $y$, $\tilde{y}$ and the unit tangent vector of this geodesic at $y$.
  Obviously, $w' \in P$.
  
  Denote the curve in $\Sigma$ from $y$ to $w'$ by $\gamma$ and the geodesic
  from $y$ to $w'$ in $P$ by $\gamma_0$. Since $\langle w', \tilde{y} \rangle
  = \langle y, \tilde{y} \rangle = 0$, let $z$ be the point $z \in \gamma$
  such that the function $\langle \gamma (t), \tilde{y} \rangle$ of $t$
  achieves its minimum.
  
  Again, we assume that $z$ is a smooth point of $\Sigma$. Let $v$ the normal
  to $\gamma$ in $P$. We know that $\gamma_0$ and $\gamma$ bounds a region
  $\Omega$ in $P$, we fix the orientation of $v$ so that it is consistent with
  $\tilde{y}$, that is $v$ points into $\Omega$.
  
  We change coordinates so that $e_0 = y$, $\tilde{y} = e_1$ and a unit
  tangent vector $e_2$ at $y$. The set $\{e_0, e_1, e_2 \}$ is an orthonormal
  basis of $P$. We can identify $P$ with $\mathbb{R}^{2, 1}$. The region
  $\Omega$ is located in the strip
  \[ \{z \in P =\mathbb{R}^{2, 1} : \langle z, \tilde{y} \rangle \leqslant z^1
     \leqslant 0\} . \]
  Since $v$ points into $\Omega$, so $v$ cannot point to the outward of
  $\Omega$ which is $- e_1$ direction. Hence $\langle v, \tilde{y} \rangle =
  \langle v, e_1 \rangle > 0$. Since $\langle \gamma (t), \tilde{y} \rangle$
  achieves minimum at $z$, $\tilde{y}$ is normal to $\gamma'$ at $z$. We can
  assume that
  \[ \tilde{y} = a_0 z + a_1 v. \]
  From $\langle \tilde{y}, z \rangle < 0$, $z$ is time-like, so $a_0 > 0$; and
  from $\langle v, \tilde{y} \rangle > 0$, $a_1 > 0$. Actually $a_1^2 = a_0^2
  + 1$, but it is not needed. Let $w$ be a point very close to $z$, since
  $\langle \gamma (t), \tilde{y} \rangle$ achieves minimum at $z$, so $\langle
  w - z, \tilde{y} \rangle \geqslant 0$. Hence,
\begin{align}
0 \leqslant & \langle w - z, \tilde{y} \rangle \\
= & \langle w - z, a_1 v \rangle + \langle w - z, a_0 z \rangle
\\
= & a_1 \langle w - z, v \rangle + a_0 (\langle w, z \rangle + 1)
\\
= & a_1 \langle w, v \rangle + a_0 (\langle w, z \rangle + 1) .
\end{align}
  From the distance formula \eqref{distance}, $\langle w, z \rangle + 1 < 0$.
  By positivity of $a_0$ and $a_1$,
  \[ \langle w, v \rangle > 0. \]
  Now we center the coordinate system at $z$, let $z = e_0$, unit tangent
  vector $e_1$ at $z$ and $v = e_2$. Any point $w$ in $P$ satisfies $\langle
  w, e_i \rangle = 0$ for $i \geqslant 3$. Since $\langle w, v \rangle > 0$
  near $z$ and the normal $\tilde{z}$ in $\mathbb{H}^{n + 1}$ is a linear
  combination of $e_i$ for $i \geqslant 2$, so
  \[ \langle w, \tilde{z} \rangle = \langle w, v \rangle \langle \tilde{z}, v
     \rangle . \]
  There is a curve going over $z$ in $P$ to meet the point $w'$, so $v$ must
  points outward of the region bounded by $M$, so $\langle \tilde{z}, v
  \rangle > 0$. Therefore, we get $\langle w, \tilde{z} \rangle > 0$. However,
  because of convexity near $z$ and that $\langle w, \tilde{z} \rangle < 0$,
  it leads to a contradiction.
  
  Now the same proof works through if $y$ is not a smooth point. If $\Sigma$
  is not smooth at $z$, it is sufficient that we invoke Lemma \ref{2D local
  separation} instead.
\end{proof}

Let $(\partial M)^{\ast}$ be the convex hull of $\partial M$ in $\partial B$
and $S^{\ast}$ be the set of $\omega \in \mathbb{S}^n$ such that $e_0 \cosh
\rho_0 + \omega \sinh \rho_0 \in (\partial M )^{\ast}$. Define the following
two sets:
\[ C_1 = \{z : z = e_0 \cosh \rho + \omega \sinh \rho, \rho \geqslant 0,
   \omega \in S^{\ast} \}, \]
and
\[ C_2 = \cap_{y \in \partial M} \{z : \langle z, \tilde{y} \rangle \leqslant
   0\} . \]
\begin{lemma}
  The two sets $C_1$ and $C_2$ are equivalent:
  \[ C_1 = C_2 . \]
\end{lemma}

\begin{proof}
  We only show that $C_1 \subset C_2$ and the reverse relation is similarly
  proved. We define
  \[ \gamma_{\omega} (t) = e_0 \cosh t + \omega \sinh t, t \geqslant 0 \]
  be the radial geodesic in $\mathbb{H}^{n + 1}$ starting from $e_0$. Let $z =
  \gamma_{\omega} (t)$ be a point in $C_1$ and $x = \gamma_w (\rho_0)$. Since
  $x$ and $y$ are on the same level set of $z^0$, and $\tilde{y}$ has no $e_0$
  component, by results from convex geometry of the sphere
  {\cite{ferreira-projections-2013}}, we have that $\langle x - y, \tilde{y}
  \rangle \leqslant 0$. So $\langle x, \tilde{y} \rangle \leqslant 0$ and
  \[ \langle z, \tilde{y} \rangle = \sinh t \langle \xi, \tilde{y} \rangle =
     \tfrac{\sinh t}{\sinh \rho_0} \langle \sinh \rho_0, \tilde{y} \rangle =
     \tfrac{\sinh t}{\sinh \rho_0} \langle x, \tilde{y} \rangle \leqslant 0.
  \]
  Hence $z \in C_2$.
\end{proof}

It follows immediately that

\begin{corollary}
  We have that $M \subset C_1 = C_2$ and $e_0 \not\in M$.
\end{corollary}

\begin{proof}
  The fact $M \subset C_2$ is obvious. If $e_0$ is a point in $M$, it has to
  be an interior point of $M$, and $M$ is smooth here. But $C_1$ is a convex
  radial cone and $M \subset C_1$ says that $M$ is not smooth at $e_0$.
\end{proof}

Let $\hat{M}$ be the enclosed region by $\partial B$ and $M$ with the unit
normals $\nu$ pointing outside of $\hat{M}$. Let
\[ C_0 = \cap_{z \in \partial M} \{w \in \mathbb{H}^{n + 1} : \langle w,
   \tilde{z} \rangle \leqslant 0\} . \]
We see that $\hat{M}$ is contained in $C_0$.

\begin{lemma}
  \label{height}Let $\rho > 0$ and $C \subset \mathbb{R}^{n + 1, 1}$ be the
  cone
  \[ C = \{z = e_0 \cosh t + \omega \sinh t : t \geqslant 0, \xi \in S^{\ast}
     \} \]
  for some convex set $S^{\ast} \subset \mathbb{S}^n$. If for all $\omega \in
  S^{\ast}$ and some $\varepsilon \in \left( 0, \tfrac{\pi}{2} \right)$ such
  that $\langle \omega, e_1 \rangle \geqslant \cos (\tfrac{\pi}{2} -
  \varepsilon)$,
  \[ z^1 = \langle z, e_1 \rangle \geqslant \frac{\sinh R}{\cos \varepsilon}
  \]
  if $B_R (z) \subset C$.
\end{lemma}

\begin{proof}
  We define a new cone
  \[ C' = \{z = e_0 \cosh t + \omega \sinh t : t \geqslant 0, \langle \omega,
     e_1 \rangle \geqslant \cos (\tfrac{\pi}{2} - \varepsilon)\} . \]
  Obviously, $C \subset C'$ and the distance from $z$ to $\partial C'$ denoted
  by $R'$ is greater than $R$. Assume that $\partial B_{R'} (z)$ touches
  $\partial C'$ at $D$ and let $A$ represents the point $z$. The geodesic
  starting from $A$ to $D$ must be orthogonal to the ray $O D$. The three
  points $O$, $A$ and $D$ spans a two dimensional hyperbolic subspace. By the
  hyperbolic sine law (see {\cite[Chapter 1]{katok-fuchsian-1992}}), we have
  that
  \[ \tfrac{\sinh R'}{\sin \angle D O A} = \tfrac{\sinh O A}{\sin
     \tfrac{\pi}{2}}, \]
  where $O A = \tfrac{2 |x|}{1 - |x|^2}$. We have that
  \[ x^1 = |x| \cos \angle A O B', \]
  where $B'$ is point realizing the shortest Euclidean distance from $A$ to
  the line $x^1$ axis. So
  \[ z^1 = \tfrac{2 x^1}{1 - |x|^2} = \tfrac{2 |x| \cos \angle A O B'}{1 -
     |x|^2} = \tfrac{\cos \angle A O B'}{\sin \angle D O A} \sinh R' . \]
  We just have to estimate $\tfrac{\cos \angle A O B'}{\sin \angle D O A}$:
\begin{align}
\tfrac{\cos \angle A O B'}{\sin \angle D O A} = & \tfrac{\cos
(\tfrac{\pi}{2} - \varepsilon - \angle D O A)}{\sin \angle D O A}
\\
= & \tfrac{\sin (\varepsilon + \angle D O A)}{\sin \angle D O A}
\\
= & \sin \varepsilon \cot \angle D O A + \cos \varepsilon .
\end{align}
  The right hand of the the above is uniformly bounded below by
  $\tfrac{1}{\cos \varepsilon}$ since $\angle D O A = \tfrac{\pi}{2} -
  \varepsilon$ and
  \[ \sin \varepsilon \cot \angle D O A + \cos \varepsilon = \sin \varepsilon
     \cot (\tfrac{\pi}{2} - \varepsilon) + \cos \varepsilon = \tfrac{1}{\cos
     \varepsilon} . \]
  This concludes our proof.
\end{proof}

We can also look at this in the Poincar{\'e} ball model. Then if $B_R (A)
\subset C$, in terms of ball model,
\[ z^1 := \tfrac{2 x^1}{1 - |x|^2} \geqslant \sinh R (1 + \delta), \]
($x$ is the point of $A$) where $\delta$ only depends only on $\varepsilon$.
If $x^1 / |x| = \cos \theta \in [0, \tfrac{\pi}{2} - \varepsilon]$, we have a
lower bound on $x^1$.

\begin{corollary}
  Under the same conditions of Lemma \ref{height}, we have that there exists a
  $\delta' \in (0, 1)$ such that
  \[ x^1 \geqslant \delta' > 0. \]
\end{corollary}

\begin{proof}
  We solve $x^1$ in terms of $z^1$, we got
  \[ z^1 \left( \tfrac{(x^1)^2}{\cos^2 \theta} - 1 \right) + 2 x^1 = 0, \]
  and
  \[ x^1 = \frac{- 2 + \sqrt{4 + 4 z^2 / \cos^2 \theta}}{2 \tfrac{z}{\cos^2
     \theta}} = \frac{\sqrt{1 + z^2 / \cos^2 \theta} - 1}{z / \cos^2 \theta} =
     \frac{z}{\sqrt{1 + z^2 / \cos^2 \theta} + 1}, \]
  (we have dropped the negative solution) since the function
  $\tfrac{x}{\sqrt{1 + x^2} + 1}$ is increasing, we have that $x^1$ is bounded
  below by the above when $z = \sinh R (1 + \delta)$.
\end{proof}

\begin{lemma}
  \label{N0 less than 0}For every interior point $z$ of $M$, we have that the
  zeroth component of the normal at $x$ is strictly negative i.e.
  \[ \tilde{z}^0 < 0. \]
\end{lemma}

\begin{proof}
  Let $w (t)$ be the geodesic ray starting from $e_0$ and passing through $z$.
  Note $z$ is an interior point of $\Sigma$. So we have that
  \[ \langle \tilde{z}, w' (t_0) \rangle < 0 = \langle w (t_0), \tilde{z}
     \rangle . \]
  The $\langle \tilde{z}, w' (t_0) \rangle < 0$ follows from convexity of
  $\Sigma$ and $\langle w (t_0), \tilde{z} \rangle = 0$ follows from the fact
  that $\tilde{x}$ is in the tangent space at $z = w (t_0)$.
  
  We use the unit parametrization $w (t) = e_0 \cosh t_0 + \xi \sinh t_0$
  where $\xi \in \mathbb{S}^n$ for the curve $w (t)$. The above condition are
  then
\begin{align}
\langle w' (t_0), z \rangle & = \langle e_0, \tilde{z} \rangle \sinh t_0 +
\langle \xi, \tilde{z} \rangle \cosh t_0 < 0 \\
\langle w (t_0), \tilde{z} \rangle & = \langle e_0, \tilde{z} \rangle
\cosh t_0 + \langle \xi, \tilde{z} \rangle \sinh t_0 = 0.
\end{align}
  Eliminating $\langle \xi, \tilde{z} \rangle$, we have
\begin{align}
& \langle w' (t_0), \tilde{z} \rangle \\
= & \langle e_0, \tilde{z} \rangle (\sinh t_0 - \tfrac{\cosh^2 t_0}{\sinh
t_0}) \\
= & - \tfrac{1}{\sinh t_0} \langle e_0, \tilde{z} \rangle =
\tfrac{\tilde{z}^0}{\sinh t_0} < 0.
\end{align}
  So the zeroth component is less than zero.
\end{proof}

The following corollary shows that $M$ lies in a half geodesic ball omitting
one point in the direction of the axis. It is an important fact in writing the
inverse mean curvature flow in coordinates and turning it into a parabolic
partial differential equation of a scalar function with Neumann boundary
conditions.

\begin{corollary}
  \label{M in pointed half ball}Let $p_1 = e_0 \cosh \rho_0 + e_1 \sinh
  \rho_0$, then
  \[ M \subset B^+ \backslash \{p_1 \} . \]
\end{corollary}

\begin{proof}
  If the interior of $M$ touches the geodesic sphere $\partial B_0$ at $p_1
  \in M$, then $p_1$ is an interior maximum point of the function $z^0$. Since
  \[ \Delta z^0 = n z^0 - H \nu^0 . \]
  From Lemma \ref{N0 less than 0} follows $\Delta \rho > 0$ at $p_1$
  contradicting the fact $p$ is maximum point of $\rho$.
\end{proof}

\begin{lemma}
  \label{estimate normal}The first component of $\nu$ is bounded, in
  particular,
  \[ \nu_1 \leqslant - c_0, \]
  where $c_0$ depends only on the distance of $S^{\ast}$ to the equator
  \begin{equation}
    \mathcal{H} (e_1) = \{e \in \mathbb{S}^n : \langle e, e_1 \rangle = 0\} .
    \label{equator}
  \end{equation}
\end{lemma}

\begin{proof}
  Any point in $(\partial M)^{\ast}$ can be represented by
  \[ z' = e_0 \cosh \rho_0 + \xi' \sinh \rho_0, t > 0, \xi' \in S^{\ast} . \]
  By \eqref{position normal relation}, $\langle \nu, z' \rangle \leqslant 0$.
  From Lemma \ref{N0 less than 0}, we can assume that
  \[ \nu = - e_0 \sinh s + \xi \cosh s, s \geqslant 0, \xi \in \mathbb{S}^n .
  \]
  We have that
  \[ \langle \nu, z' \rangle = \sinh s \cosh \rho_0 + \langle \xi, \xi'
     \rangle \sinh \rho_0 \cosh s \leqslant 0. \]
  So
  \[ \langle \xi, \xi' \rangle \leqslant - \tfrac{\tanh s}{\tanh \rho_0} . \]
  Note that
  \[ \langle \xi, \xi' \rangle \leqslant 0, - 1 \leqslant - \tfrac{\tanh
     s}{\tanh \rho_0} . \]
  This gives a trivial bound $0 \leqslant s \leqslant \rho_0$ on $s$. And
  since $\xi'$ lies in a convex set $S^{\ast}$ of $\mathbb{S}^n$, this gives a
  bound on $\langle \xi, e_1 \rangle$ which depends on the distance of
  $S^{\ast}$ to the equator $\mathcal{H} (e_1)$.
\end{proof}

We now prove an estimate of $z^1$ for convex free boundary hypersurfaces.

\begin{lemma}
  \label{height estimate}Let $M$ be a strictly convex hypersurfaces, then
  \[ z^1 \geqslant \delta > 0 \]
  for all $M$, where the constant $\delta$ depends only on $\sup_M |A|$ and
  the distance of $\partial M$ to the equator $\mathcal{H} (e_1)$.
\end{lemma}

\begin{proof}
  Let $z \in M$ be a global minimum point of $z^1$. Let $e_1' = e_1 + \langle
  e_1, z \rangle z$. The vector $e'$ is a vector in $T_z \mathbb{H}^{n + 1}$,
  moreover $\langle e_1', e_1' \rangle = 1 + (z^1)^2 > 0$. Let now $e = e_1' /
  \sqrt{\langle e_1', e_1' \rangle}$.
  
  Since $z$ is global minimum, $e$ must be a normal vector of $M$ pointing
  inward of $\hat{M}$. Let $P$ be the hyperbolic subspace normal to $e$. Due
  to the previous lemma, it is possible to write $M$ as a graph around $z$
  over $P$. The graph function satisfies
  \[ w_{i j} = - h_{i j} \langle \nu, e \rangle + g_{i j} \langle w, e \rangle
     . \]
  We only have to show that $\langle \nu, e \rangle$ is bounded by a negative
  constant. Indeed,
\begin{align}
& \langle \nu, e \rangle \\
= & \tfrac{1}{\sqrt{\langle e_1', e_1' \rangle}} \langle \nu, e_1 +
\langle e_1, z \rangle z \rangle \\
= & \tfrac{1}{1 + (z^1)^2} \langle \nu, e_1 \rangle + \tfrac{z^1}{1 +
(z^1)^2} \langle \nu, z \rangle .
\end{align}
  By Lemma \ref{convexity via position normal relation}, $\langle \nu, z
  \rangle \leqslant 0$. So
  \[ \langle \nu, e \rangle \leqslant \tfrac{1}{1 + (z^1)^2} \langle \nu, e_1
     \rangle \leqslant \tfrac{- c_0}{1 + (z^1)^2}, \]
  where the constant $c_0$ is from Lemma \ref{estimate normal}. Define
  \[ \hat{M}' = \cap_{y \in M} \{z \in \mathbb{H}^{n + 1} : \langle z,
     \tilde{y} \rangle \leqslant 0\} . \]
  By the previous consideration, we have that $\hat{M}'$ satisfies an interior
  sphere condition at $z$ with radius of the ball depending only on the second
  fundamental form and $\langle \nu, e \rangle$. Due to $B_R (z) \subset
  \hat{M}' \subset C_1 = C_2$, from Lemma \ref{height}, we have that
  \[ z^1 \geqslant \delta > 0. \]
  
\end{proof}

\section{Moebius coordinates}\label{moebius coordinates}

Lambert and Scheuer {\cite{lambert-inverse-2016}} introduced the
\text{{\itshape{Moebius coordinates}}} for free boundary hypersurfaces in the
Euclidean unit ball.

\begin{definition}
  Let $D \subset \mathbb{R}^n$ be the unit disk. Define the Moebius
  coordinates for the pointed half ball
  \[ \mathbb{B}^+ =\mathbb{B}_1^+ (0) \backslash \{\partial_1 \} \]
  to be the diffeomorphism $f : D \times [1, \infty) \to \mathbb{B}^+$,
  \[ f (\xi, \lambda) = \tfrac{4 \lambda \xi + (1 + | \xi |^2) (\lambda^2 - 1)
     \partial_1}{(1 + \lambda^2) + (1 - \lambda)^2 | \xi |^2} . \]
\end{definition}

In this section, we generalized the Moebius coordinates of convex free
boundary hypersurfaces by {\cite{lambert-inverse-2016}} to the hyperbolic
case. To this end, we will make use of the ball model. In this way, the
Moebius coordinate is just composition with one more scaling with a scaling
factor $r_0$ defined by the relation $\tfrac{1 + r_0^2}{1 - r_0^2} = \cosh
\rho_0$. The Moebius coordinates are then given by the diffeomorphism
\[ \psi : \mathbb{B}^n \times [1, \infty) \to \mathbb{B} \to \mathbb{B}^{n +
   1} (r_0), r_0 < 1 \]
sending $(\xi, \lambda)$ to
\[ \psi (\xi, \lambda) = r_0 f (\xi) = r_0 \tfrac{4 \lambda \xi + (1 + | \xi
   |^2) (\lambda^2 - 1) \partial_1}{(1 + \lambda)^2 + (1 - \lambda)^2 | \xi
   |^2} . \]
The metric components of the metric $b$ under this new coordinate are $b
(\tfrac{\partial}{\partial \xi^i}, \tfrac{\partial}{\partial \xi^j})$, $b
(\tfrac{\partial}{\partial \xi^i}, \tfrac{\partial}{\partial \lambda})$ and $b
(\tfrac{\partial}{\partial \lambda}, \tfrac{\partial}{\partial \lambda})$. It
is easy to see that
\[ b (\tfrac{\partial}{\partial \xi^i}, \tfrac{\partial}{\partial \xi^j}) =
   \tfrac{4}{(1 - r_0^2 |f|^2)^2} \tfrac{\partial (r_0 f)}{\partial \xi^i}
   \cdot \tfrac{\partial (r_0 f)}{\partial \xi^j} = \frac{4 r_0^2}{(1 - r_0^2
   |f|^2)^2} \tfrac{\partial f}{\partial \xi^i} \cdot \tfrac{\partial
   f}{\partial \xi^j} \]
and similar formulas hold for other components. The dot $\cdot$ here
represents the Euclidean inner product. These components can be explicitly
calculated although what is really needed is $b (\tfrac{\partial}{\partial
\xi^i}, \tfrac{\partial}{\partial \lambda}) = 0$.

\begin{lemma}
  \label{lemma metric mobius}Let $\phi = (1 + \lambda)^2 + (1 - \lambda)^2 |
  \xi |^2$. The metric is
  \begin{equation}
    b = \phi_1^2 \mathrm{d} \lambda^2 + \phi_2^2 \delta_{i j},
  \end{equation}
  where
  \begin{equation}
    \phi_1 = \frac{4 (1 + | \xi |^2)}{\phi (1 - r_0^2 |f|^2)}, \phi_2 =
    \frac{8 r_0 \lambda}{\phi (1 - r_0^2 |f|^2)} . \label{metric components in
    Moebius coordinates}
  \end{equation}
\end{lemma}

\begin{proof}
  By direct but tedious calculation. The components $b
  (\tfrac{\partial}{\partial \xi^i}, \tfrac{\partial}{\partial \lambda}) = 0$
  also follows from {\cite{lambert-inverse-2016}}.
\end{proof}

\begin{proposition}
  \label{graphical}Let $y : M \to (\mathbb{B}^{n + 1}, b)$ the the embedding
  of a strictly convex free boundary hypersurface $M$. Then $M$ can be written
  as a graph in Moebius coordinates around $\partial_1$ that is
  \[ x = \psi (\xi, u (\xi)) . \]
\end{proposition}

Before we prove this proposition, we give the following lemma as a preparation
where the idea of the proof is also used in proof of Proposition
\ref{graphical}.

\begin{lemma}
  \label{ei N less than 0}Suppose that $M$ is convex with free boundary in
  $B_0$, in the ball model we have
  \[ \langle \partial_1, N \rangle < 0, \]
  where $N$ is the unit normal to $M$ in the Poincar{\'e} ball model.
\end{lemma}

\begin{proof}
  Since $M$ is convex and that $\partial M$ is convex in $\partial B_0$ or
  equivalently in $\partial \mathbb{B} (r_0)$, $\langle \partial_1, N \rangle
  < 0$ along $\partial M$. Assume on the contrary that $\langle \partial_1, N
  \rangle \geqslant 0$, then the maximum of $\langle \partial_1, N \rangle$
  occurs in the interior of $M$. Let $y = y^i \partial_i$ be any tangent
  vector fields on $M$, at a maximum point $p$ of $\langle \partial_1, N
  \rangle$,
  \[ D_y \langle \partial_1, N \rangle = \langle N, D_y \partial_1 \rangle +
     \langle \partial_1, D_y N \rangle = 0. \]
  Since
  \[ D_j \partial_i = z^j \partial_i + z^i \partial_j - \delta_{i j} \hat{z},
  \]
  so
  \begin{equation}
    \langle N, D_y \partial_1 \rangle = y_j z^j \langle \partial_1, N \rangle
    - y_1 \langle z, N \rangle =: \langle Y, y \rangle, \label{ei derivative}
  \end{equation}
  where $Y = u^{- 2} \hat{z} \langle \partial_1, N \rangle - u^{- 2} \langle
  \hat{z}, N \rangle \partial_1$. So
  \[ \langle \partial_1, D_y N \rangle = - \langle y, Y \rangle \]
  at $p$. Note that $Y$ is orthogonal to $N$, letting $\partial_1^{\top}$ be
  the projection of $\partial_1$ to the tangent space of $M$, so
  \[ A (\partial_1^{\top}) = - Y, \]
  and
  \[ \langle A (\partial_1^{\top}), \partial_1^{\top} \rangle = - \langle Y,
     \partial_1 \rangle = u [\langle x, N \rangle - x^1 \langle \partial_1, N
     \rangle] < 0. \]
  Due the to the assumption at $p$, $\langle \partial_1, N \rangle \geqslant
  0$ and recall that $x^1 > 0$, so
  \[ \langle A (\partial_1^{\top}), \partial_1^{\top} \rangle < 0. \]
  However $M$ is convex, $A$ is positive. The contradiction leads to $\langle
  \partial_1, N \rangle < 0$.
\end{proof}

Now we turn to the proof of Proposition \ref{graphical}.

\begin{proof}[Proof of Proposition \ref{graphical}]
  First, by Lemma \ref{M in pointed half ball} Moebius coordinate is well
  defined on $M$. By implicit function theorem, we have to show that
  \[ \langle \tfrac{\partial \psi}{\partial \lambda}, N \rangle < 0. \]
  Note that
  \begin{equation}
    \tfrac{\partial \psi}{\partial \lambda} = r_0 \tfrac{\partial f}{\partial
    \lambda} = r_0 \frac{(1 + | \xi |^2) (1 - \lambda^2)}{\lambda ((1 +
    \lambda)^2 + (1 - \lambda)^2 | \xi |^2)} (f - \tfrac{\lambda^2 +
    1}{\lambda^2 - 1} \partial_1) .
  \end{equation}
  Since $\lambda > 1$ by Lemma \ref{height estimate} and $x = r_0 f$, we are
  reduce to show
  \[ \zeta (x) : = \langle x - r_0 \tfrac{1 + \lambda^2}{\lambda^2 - 1}
     \partial_1, N \rangle > 0. \]
  Now we follow same lines of arguments as in Lemma \ref{ei N less than 0} to
  show the above inequality. Suppose on the contrary that $\zeta (x) \leqslant
  0$. Since $\langle x, N \rangle$ vanishes along $\partial M$,
  \[ \zeta (x) = - \langle r_0 \tfrac{1 + \lambda^2}{\lambda^2 - 1}
     \partial_1, N \rangle > 0 \]
  along $\partial M$. So the minimum of $\zeta$ occurs at an interior point
  $p$ of $M$ and at $p$, $\nabla \zeta = 0$. Letting $\Lambda = r_0 \tfrac{1 +
  \lambda^2}{\lambda^2 - 1}$ and $X = x - \Lambda \partial_1$, then at $p$,
  \[ \langle D_y X, N \rangle + \langle X, D_y N \rangle = 0. \]
  Now note that $\langle D_y x, N \rangle$ vanishes. Indeed, this follows from
  $D_j (x^i \partial_i) = \partial_j + x^i D_j \partial_i$, and
\begin{align}
D_y x & = y + y^j x^l D_j e_l \\
& = y + y^j x^i (z^j \partial_i + z^i \partial_j - \delta_{i j} \hat{z})
\\
& = y + y^j z^j x + y x^i z^i - x^i y^i \hat{z} \\
& = y (1 + x^i z^i) .
\end{align}
  So
  \[ \langle X, D_y N \rangle = - \langle D_y X, N \rangle = \Lambda \langle
     D_y \partial_1, N \rangle = \Lambda \langle y, Y \rangle \]
  which gives
  \[ A (X^{\top}) = \Lambda Y, \]
  again here $Y = u^{- 2} \hat{z} \langle \partial_1, N \rangle - u^{- 2}
  \langle \hat{z}, N \rangle \partial_1$.
  
  We compute now
\begin{align}
& \langle A (X^{\top}), X^{\top} \rangle \\
= & \langle \Lambda Y, X \rangle \\
= & \tfrac{\Lambda}{u^2} \langle x \langle \partial_1, N \rangle - \langle
x, N \rangle \partial_1, x - \Lambda \partial_1 \rangle \\
= & \Lambda [|x|^2 \langle \partial_1, N \rangle + \langle x, N \rangle
(\Lambda - x^1) - \Lambda x^1 \langle \partial_1, N \rangle] .
\end{align}
  Since at $p$, $\zeta (x) \leqslant 0$ implies that $\langle x, N \rangle
  \leqslant \Lambda \langle \partial_1, N \rangle$ and for all of $M$,
  $\Lambda - x^1 > 0$, we have
\begin{align}
\langle A (X^{\top}), X^{\top} \rangle \leqslant & \Lambda [|x|^2 \langle
\partial_1, N \rangle + \Lambda (\Lambda - x^1) \langle \partial_1, N
\rangle - \Lambda x^1 \langle \partial_1, N \rangle] \\
= & \Lambda \langle \partial_1, N \rangle (|x|^2 + \Lambda^2 - 2 \Lambda
x^1) \\
= & \langle \partial_1, N \rangle \Lambda \left( \sum_2^{n + 1} |x^i |^2 +
(\Lambda - x^1)^2 \right) .
\end{align}
  The factor in the big bracket on the left is obviously positive, so $\langle
  A (X^{\top}), X^{\top} \rangle < 0$ by the previous lemma. This again
  contradicts the convexity of $M$.
\end{proof}

Now we can reduce the inverse mean curvature flow to a scalar parabolic
equation.

\begin{theorem}
  Let $F$ be a solution of \eqref{imcf} for some $[0, \varepsilon)$ such that
  for all $t \in [0, \varepsilon)$ the flow $M_t$ are graphs in Moebius
  coordinates in the pointed half ball $B_0^+ \backslash p_1$, that is
  \[ M_t = \{(\xi (t, \omega), u (t, \xi)) : (t, \omega) \in [0, \varepsilon)
     \times \mathbb{D}\} . \]
  Adopting the notations in Lemma \ref{lemma metric mobius}, then $u$ solves a
  parabolic partial differential equation with Neumann boundary value
  condition:
  \begin{equation}
    \left\{\begin{array}{lll}
      \tfrac{\partial}{\partial t} u & = - \tfrac{v}{H} & \text{ in } (0,
      \varepsilon) \times \mathbb{D}\\
      \sum_i \xi^i \tfrac{\partial u}{\partial \xi^i} & = 0 & \text{ on } [0,
      \varepsilon) \times \partial \mathbb{D}\\
      u (\cdot, 0) & = u_0, & 
    \end{array}\right. \label{graphical imcf}
  \end{equation}
  where $v$ is defined to be
  \begin{equation}
    v = \tfrac{1}{4} \phi (1 - r_0^2 |f|^2) \sqrt{\tfrac{1}{(1 + | \xi |^2)^2}
    + \tfrac{1}{4 r_0^2 \lambda^2} \sum_i (\tfrac{\partial u}{\partial
    \xi^i})^2} .
  \end{equation}
\end{theorem}

\begin{proof}
  Let $X = (\xi, \lambda)$, because of Proposition \ref{graphical}, the
  inverse mean curvature flow is equivalent to
  \begin{equation}
    (\partial_t X)^{\bot} = \tfrac{1}{H} \nu
  \end{equation}
  up to tangential diffeomorphisms where $\nu$ points downward of
  $\lambda$-direction, so
  \[ \tfrac{1}{H} = \langle \partial_t X, \nu \rangle = \langle
     \tfrac{\partial}{\partial \lambda}, \nu \rangle \tfrac{\partial}{\partial
     t} u. \]
  We write $b_{\lambda \lambda} = \phi_1^2$ and $b_{i j} = \phi_2^2 \delta_{i
  j}$, the the vector $\tilde{\nu} = \phi_1^2 \partial u - \phi_2^2
  \partial_{\lambda}$ is normal to the graph, where $\partial u = \sum_i
  \tfrac{\partial u}{\partial \xi^i} \tfrac{\partial}{\partial \xi^i}$. So
  \begin{equation}
    \langle \tfrac{\partial}{\partial \lambda}, \nu \rangle = \tfrac{1}{|
    \tilde{\nu} |} \langle \tfrac{\partial}{\partial \lambda}, \tilde{\nu}
    \rangle = - \tfrac{1}{| \tilde{\nu} |} \phi^2_1 \phi_2^2 .
  \end{equation}
  We see that $v = \tfrac{| \tilde{\nu} |}{\phi_1^2 \phi_2^2}$, and the
  equation \eqref{graphical imcf} now easily follows from $\partial_t \xi^i
  \equiv 0$.
\end{proof}

We partially calculate the mean curvature $H$ in terms of $u$ and the metric
in \eqref{metric components in Moebius coordinates}.

\begin{lemma}
  Let $\tilde{\nu}$, $v$ be as in the previous theorem, the mean curvature $H$
  of a graph of some function $u$ in Moebius coordinates is given by
  \begin{equation}
    H = \tfrac{1}{v} g^{i j} \partial_i \partial_j u + F', \label{mean
    curvature expansion}
  \end{equation}
  where $F'$ contains lower order terms and depends on $x$, $u$ and $\partial
  u$.
\end{lemma}

\begin{proof}
  Let the inverse metric of the graph be $g^{i j}$, so the mean curvature of a
  graph of the function $u$ is given by
  \begin{equation}
    H = g^{i j} \tfrac{1}{| \tilde{\nu} |} \langle D_{\partial_i + u_i
    \partial_{\lambda}} (\phi_1^2 \partial u - \phi_2^2 \partial_{\lambda}),
    \partial_j + u_j \partial_{\lambda} \rangle .
  \end{equation}
  The terms contains $\partial_i \partial_j u$ is $\tfrac{1}{| \tilde{\nu} |}
  \langle D_{\partial_i} (\phi_1^2 \partial u), \partial_j \rangle$, it is
  easy to see \eqref{mean curvature expansion} holds.
\end{proof}

\section{Convergence to totally geodesic disks}\label{important estimates}

We define $\bar{T}$ to be the largest time such that for all $t \in [0,
\bar{T})$ the hypersurface $M_t$ is strictly convex and $T^{\ast}$ be the
largest time such that the flow exists. The for all $t \in [0, T^{\ast})$,
$M_t$ is mean convex. Obviously, $\bar{T} \leqslant T^{\ast}$.

\begin{proposition}
  Let $F (t, \xi)$ be a solution to the inverse mean curvature flow on the
  interval $[0, \bar{T})$. Then the principal curvature are bounded, that is
  \begin{equation}
    \sup_{M_t} |A| \leqslant C, \label{curvature bound}
  \end{equation}
  for all $t \in [0, \bar{T})$ where $C$ depends only on the radius $\rho_0$,
  initial position of $M_0$ and the mean curvature of the initial hypersurface
  $M_0$.
\end{proposition}

\begin{proof}
  We consider the auxiliary function $\zeta = \log H + \log z^0$. From
  \eqref{normal derivative H} and \eqref{normal derivative z},
  \[ \nabla_{\eta} \zeta = \nabla_{\eta} \log H + \nabla_{\eta} \log z^0 = -
     \coth \rho_0 + \tanh \rho_0 < 0. \]
  It follows that the maximum of $\zeta$ is not attained at the boundary
  $\partial M$. Using the evolutions \eqref{evolution log H} and
  \eqref{evolution f z},
\begin{align}
& \partial_t \zeta - \tfrac{1}{H^2} \Delta \zeta \\
= & (\partial_t - \tfrac{1}{H^2} \Delta) (\log H + \log z^0) \\
= & - \tfrac{| \nabla H|^2}{H^4} - \tfrac{|A|^2}{H^2} + \tfrac{2}{H z^0}
\nu_0 + \tfrac{1}{H^2 {z^0}^2} | \nabla z^0 |^2 .
\end{align}
  We rewrite the above evolution equation as
\begin{align}
& (\partial_t - \tfrac{1}{H^2} \Delta) \zeta + \tfrac{1}{H^2} \langle
\tfrac{\nabla H}{H} - \tfrac{\nabla z^0}{z^0}, \nabla \zeta \rangle = -
\tfrac{|A|^2}{H^2} + \tfrac{2}{H z^i} \nu^i .
\end{align}
  The right hand side is less than zero by \eqref{N0 less than 0}. From
  maximum principle,
  \[ \zeta = \log H + \log z^0 \leqslant \sup_{M_0} \zeta = \sup_{M_0} (\log H
     + \log z^0) . \]
  Since for each $t \in [0, \bar{T})$, $M_t$ is convex, each principal
  curvature is less than the mean curvature, so we obtain the curvature bound
  \eqref{curvature bound}.
\end{proof}

\begin{lemma}
  There exists a $C^{1, \alpha}$ limiting hypersurface $\partial M_{\bar{T}}$
  arising as the limit of $\partial M_t$. The boundary $\partial M_{\bar{T}}$
  is either an equator of the sphere $\partial B_0$ or is contained in an open
  hemisphere of $\partial B_0$.
\end{lemma}

\begin{proof}
  The proof is the same with {\cite[Corollary 5]{lambert-inverse-2016}}
  relying on {\cite{makowski-rigidity-2016}}.
\end{proof}

\begin{lemma}
  \label{mean curvature lower bound}Let $M_t$ be a solution to the inverse
  mean curvature flow, if $\partial M_{\bar{T}}$ is positive distance away
  from the equator, then
  \[ \sup_{[0, \bar{T}) \times \mathbb{D}} \tfrac{1}{H} \leqslant c, \]
  where $c$ depends only on $M_0$ and the distance of $\partial M_{\bar{T}}$
  to the equator.
\end{lemma}

\begin{proof}
  Let $f (q) = - \log (\Lambda - q)$, $q = \lambda z^1 + z^0$. We require that
  $0 < \Lambda < \tfrac{1}{\cosh \rho_0}$ and $\lambda < - 1 - \tfrac{\cosh
  \rho_0}{\delta}$, where $\delta$ is the number in Lemma \ref{height
  estimate}. We have that $f' (q) = \tfrac{1}{\Lambda - q}$, $f'' (q) =
  \tfrac{1}{(\Lambda - q)^2}$.
  
  Now
  \[ (\partial_t - \tfrac{1}{H^2} \Delta) \log \tfrac{1}{H} = \tfrac{| \nabla
     H|^2}{H^4} + (- n + |A|^2) \tfrac{1}{H^2} . \]
  Evolution of $F$:
  \[ (\partial_t - \tfrac{1}{H^2} \Delta) F = - \tfrac{n}{H^2} \tfrac{\partial
     F}{\partial z^{\alpha}} z^{\alpha} + \tfrac{2}{H} \tfrac{\partial
     F}{\partial z^{\alpha}} \nu^{\alpha} - \tfrac{\partial^2 F}{H^2 \partial
     z^{\alpha} \partial z^{\beta}} \langle \nabla z^{\alpha}, \nabla
     z^{\beta} \rangle . \]
  Let $F = f \circ q (z)$, note that
  \[ \tfrac{\partial (f \circ q)}{\partial z^{\alpha} \partial z^{\beta}} =
     \tfrac{\partial}{\partial z^{\alpha}} (f' \tfrac{\partial q}{\partial
     z^{\beta}}) = f'' \tfrac{\partial q}{\partial z^{\alpha}} \tfrac{\partial
     q}{\partial z^{\beta}} + f' \tfrac{\partial^2 q}{\partial z^{\alpha}
     \partial z^{\beta}} . \]
  So we have that
\begin{align}
& (\partial_t - \tfrac{1}{H^2} \Delta) F \\
= & - \tfrac{n}{H^2} f' \tfrac{\partial q}{\partial z^{\alpha}} z^{\alpha}
+ \tfrac{2}{H} f' \tfrac{\partial q}{\partial z^{\alpha}} \nu^{\alpha} -
\tfrac{1}{H^2} f' \tfrac{\partial^2 q}{\partial z^{\alpha} \partial
z^{\beta}} \langle \nabla z^{\alpha}, \nabla z^{\beta} \rangle -
\tfrac{1}{H^2} f'' \tfrac{\partial q}{\partial z^{\alpha}} \tfrac{\partial
q}{\partial z^{\beta}} \langle \nabla z^{\alpha}, \nabla z^{\beta} \rangle
.
\end{align}
  Since $q$ is linear in $z^1$ and $z^0$, we have that the evolution of $\xi$
  \[ \xi : = \log \tfrac{1}{H} + f \circ q (z) \]
  is given by the following
  \begin{equation}
    (\partial_t - \tfrac{1}{H^2} \Delta) \xi = \tfrac{| \nabla H|^2}{H^4} -
    \tfrac{1}{H^2} f'' | \nabla q|^2 + (- n + |A|^2 - n f' q) \tfrac{1}{H^2} +
    \tfrac{2}{H} f' \tfrac{\partial q}{\partial z^{\alpha}} \nu^{\alpha} .
    \label{evolution xi}
  \end{equation}
  Now we compute the boundary derivatives of $\xi$:
\begin{align}
& \nabla_{\eta} \xi \\
= & \coth \rho_0 + \coth \rho_0 f' \tfrac{\partial q}{\partial z^1} z^1 +
f' \tfrac{\partial q}{\partial z^0} z^0 \tanh \rho_0 \\
= & \coth \rho_0 (1 + \tfrac{1}{\Lambda - q} a z^1 + \tfrac{1}{\Lambda -
q} z^0 \tanh^2 \rho_0) \\
= & \tfrac{\coth \rho_0}{\Lambda - q} (\Lambda - z^0 + z^0 \tanh^2 \rho_0)
\\
= & \tfrac{\coth \rho_0}{\Lambda - q} (\Lambda - \tfrac{1}{\cosh \rho_0})
< 0 ;
\end{align}
  where in the last line we have used along the boundary $z^0 = \cosh \rho_0$
  and $1 - \tanh^2 \rho_0 = \tfrac{1}{\cosh^2 \rho_0}$.
  
  Next we consider that the maximum of $\xi$. From the condition
  $\nabla_{\eta} \xi < 0$, we have that
  \[ \max_{[0, \bar{T}) \times \mathbb{D}} \xi = \xi (t_0, \omega_0), \]
  where $\omega_0$ is an interior point of $\mathbb{D}$. Note that $\partial_t
  \xi - \tfrac{1}{H^2} \Delta \xi \geqslant 0$ at $(t_0, \omega_0)$, by
  \eqref{evolution xi}, we have
\begin{align}
& \tfrac{| \nabla H|^2}{H^4} - \tfrac{1}{H^2} f'' | \nabla q|^2 + (- n +
|A|^2 - n f' q) \tfrac{1}{H^2} + \tfrac{2}{H} f' \tfrac{\partial
q}{\partial z^{\alpha}} \nu^{\alpha} \\
= & \tfrac{1}{H^2} ((f')^2 - f'') | \nabla q|^2 - \tfrac{n}{H^2} (1 + f'
q) + \tfrac{|A|^2}{H^2} + \tfrac{2}{H} f' \tfrac{\partial q}{\partial
z^{\alpha}} \nu^{\alpha} \\
= & 0 - \tfrac{n}{H^2} (1 + f' q) + \tfrac{|A|^2}{H^2} + \tfrac{2}{H} f'
\tfrac{\partial q}{\partial z^{\alpha}} \nu^{\alpha} \\
\geqslant & 0,
\end{align}
  where we have used the relation
  \[ 0 = \nabla \xi = - \tfrac{\nabla H}{H} + f' \nabla q \text{ at } (t_0,
     \omega_0) . \]
  Note also that
  \[ 1 + f' q = \tfrac{\Lambda}{\Lambda - q} > 0. \]
  Now we have an inequality for $\tfrac{1}{H}$ in the form
  \[ - n (1 + f' q) \left( \tfrac{1}{H} \right)^2 + 2 f' \tfrac{\partial
     q}{\partial z^{\alpha}} \nu^{\alpha} \left( \tfrac{1}{H} \right) +
     \tfrac{|A|^2}{H^2} \geqslant 0. \]
  The coefficients on $\tfrac{1}{H}$ are bounded, so this would imply a bound
  on $\tfrac{1}{H}$.
\end{proof}

\begin{theorem}
  The strict convexity of $M$ is preserved up to time $T^{\ast}$ and $\bar{T}
  = T^{\ast}$.
\end{theorem}

\begin{proof}
  If $\partial M_{\bar{T}}$ is the equator, we conclude from the Lemma
  \ref{mean curvature rigidity} that $M_{\bar{T}}$ is a totally geodesic disk
  and thus a singularity of the flow. This would yield $\bar{T} = T^{\ast}$.
  If $\partial M_{\bar{T}}$ is not the equator, then
  \[ \tfrac{1}{H} \leqslant c \text{ for all } t \in [0, \bar{T}) \]
  and by the estimate Lemma \ref{height estimate},
  \[ z^1 \geqslant \delta > 0. \]
  We consider the $\tilde{H}$ and define
  \[ \phi = \log \tilde{H} - (n + 1) \log z^1 - \alpha t, t < \bar{T}, \]
  where $\alpha$ will be chosen in dependence of $\delta$ and the initial
  data. Observe that from \eqref{normal derivative tilde H} and \eqref{normal
  derivative z},
\begin{align}
& \nabla_{\eta} \phi \\
= & \tfrac{1}{\tilde{H}} \nabla_{\eta} \tilde{H} - \tfrac{n + 1}{z^1}
\nabla_{\eta} z^1 \\
\leqslant & - \coth \rho_0 < 0,
\end{align}
  and for $0 < T < \bar{T}$, we have
  \[ \sup_{[0, T] \times \mathbb{D}} \phi = \phi (t_0, \xi_0), t_0 > 0. \]
  So $\xi_0$ is not at the boundary $\partial \mathbb{D}$. At $(t_0, \xi_0)$,
  $\nabla \phi = 0$ implies that
  \[ \tfrac{\nabla \tilde{H}}{\tilde{H}} = (n + 1) \tfrac{\nabla z^1}{z^1} .
  \]
  From \eqref{evolution harmonic mean curvature} and \eqref{evolution f z},
\begin{align}
& \partial_t \phi - \tfrac{1}{H^2} \Delta \phi \\
= & - \tfrac{|A|^2}{H^2} + \tfrac{2 n}{\tilde{H} H} + \tfrac{n^2}{H^2} +
\tfrac{1}{H^2 \tilde{H}^2} | \nabla \tilde{H} |^2 \\
& + \tfrac{2}{H^2 \tilde{H}} \left( \tfrac{1}{H} g_{i j} \tilde{h}^{i k}
\tilde{h}^{j l} \nabla_k H \nabla_l H - g_{i j} \tilde{h}^{i k}
\tilde{h}^{j p} \tilde{h}^{l q} \nabla h_{p q} \cdot \nabla h_{k l}
\right) \\
& - \tfrac{2 (n + 1)}{H z^1} \nu_1 - \tfrac{n + 1}{H^2 (z^1)^2} | \nabla
z^1 |^2 - \alpha .
\end{align}
  The term in the big bracket is negative (see the proof of {\cite[Lemma
  3.8]{urbas-expansion-1991}}), we have at $(t_0, \xi_0)$,
  \[ 0 \leqslant c + \tfrac{(n + 1)^2}{H^2 (z^1)^2} | \nabla z^1 |^2 - \alpha,
  \]
  where $c$ in the above depends on $\delta$ and the bound on $\tfrac{1}{H}$.
  
  Choosing a sufficiently large $\alpha$, this is a contradiction. Then we
  obtain that under the assumption that $\partial M_{\bar{T}}$ is not an
  equator, the supremum of $\phi$ would be decreasing, hence $\phi$ would be
  bounded up to $\bar{T}$. But then
  \[ \log \tilde{H} = \phi + (n + 1) \log z^1 + \alpha t \leqslant c + \alpha
     \bar{T}, \]
  which contradicts the definition of $\bar{T}$ at which $\tilde{H}$ would
  have to blow up provided $\bar{T} < T^{\ast}$.
\end{proof}

The estimates in this section are now sufficient to show that that the inverse
mean curvature flow converges to a totally geodesic disk. The proof follows
the same lines in {\cite{lambert-inverse-2016}}. For completeness, we outline
the proofs.

\begin{lemma}
  \label{mean curvature rigidity}If an embedded hypersurface $M \subset
  \mathbb{H}^{n + 1}$ is mean convex with free boundary in $B_0$ such that
  $\partial M$ is an equator, then $M$ is totally geodesic.
\end{lemma}

\begin{proof}
  If $\partial M$ is an equator, $\partial M$ determines a totally geodesic
  disk $P$ and a hyperbolic $n$-subspace with $P \subset \mathbb{H}^n$ which
  we denote by $\mathbb{H}^n$. We denote the center of the disk by $o$. Then
  $\mathbb{H}^{n + 1}$ is a warped product $\mathbb{H}^n \times_V \mathbb{R}$
  with metric
  \begin{equation}
    b = b' + \cosh^2 \ensuremath{\operatorname{dist}}_{\mathbb{H}^n} (o, p')
    \mathrm{d} s^2 =: b' + V^2 \mathrm{d} s^2,
  \end{equation}
  where $(p', s) \in \mathbb{H}^n \times \mathbb{R}$ and $b'$ is the standard
  metric of $\mathbb{H}^n$. If a piece of $M$ is given by a function $u$, the
  mean curvature is then
  \begin{equation}
    H =\ensuremath{\operatorname{div}} (\tfrac{V D u}{\sqrt{1 + V^2 |D u|^2}})
  \end{equation}
  with respect to the downward unit normal $\nu$. Here $D$ is the connection
  on $\mathbb{H}^n$. Note that $H$ is translation invariant in the
  $s$-direction. Write
  \[ L_s = \{(p', s) \in \mathbb{H}^{n + 1} : p' \in \mathbb{H}^n \} . \]
  Let $s_0$ be the maximum of $s$ such that $M \cap L_s$ is nonempty. Since
  $\partial M \subset L_0$, $s_0 \geqslant 0$. Let $p \in L_s \cap M$, the
  tangent space of $T_p M$ and $T_p L_s$ is the same, if $s_0 > 0$, then by
  interior maximum principle of {\cite[Lemma 1]{schoen-uniqueness-1983}},
  there is a small neighborhood of $p$ in $M$ are also in $L_s \cap M$. By
  repeating this argument at points different from $p$, $L_s \cap M$ actually
  contains all the interior points of $M$. However, $\partial M \subset L_0$
  and this would violate the continuity at $\partial M$. So $s_0 = 0$.
  
  Now assume that $p$ is a point of $\partial M$. By the free boundary
  condition, we have that the two tangent space $T_p M$ and $T_p L_0$ agrees.
  We can invoke the boundary point lemma of {\cite{schoen-uniqueness-1983}},
  we have that $L_0 \cap M$ contains a neighborhood of $p$ in $M$. Repeating
  the argument, we have that $M$ lies entirely in $L_0$ which says that $M$ is
  actually the geodesic disk $P$.
\end{proof}

\begin{remark}
  Note the lemma only uses the mean-convexity of $M$.
\end{remark}

\begin{theorem}
  At the maximal existence time $T^{\ast}$, $\partial M_{T^{\ast}}$ is an
  equator of the sphere $\partial B_0$ and $M_{T^{\ast}}$ is a totally
  geodesic disk.
\end{theorem}

\begin{proof}
  We assume that $\partial M_{T^{\ast}}$ is not an equator, then due to Lemma
  \ref{mean curvature lower bound}, we have that $\tfrac{1}{H} \geqslant c$ at
  $\partial M_{T^{\ast}}$. Let
  \[ F (x, u, \partial u, \partial^2 u) = - \tfrac{v}{H}, \]
  so
  \begin{equation}
    \tfrac{\partial F}{\partial u_{i j}} = \tfrac{v}{H^2} \tfrac{\partial
    H}{\partial u_{i j}} = \tfrac{v}{H^2} (\tfrac{1}{v} g^{i j}) =
    \tfrac{1}{H^2} g^{i j},
  \end{equation}
  follows from \eqref{mean curvature expansion}. The evolution is then
  uniformly parabolic and by regularity theory, we can extend a solution $u$
  of \eqref{graphical imcf} on $[0, T)$ where $T \in (0, T^{\ast})$ to a
  solution $u$ of \eqref{graphical imcf} on $[0, T + \varepsilon)$. The
  constant $\varepsilon > 0$ and depends only on the data. We can choose $T $
  such that $T^{\ast} - T < \varepsilon$ and therefore we extend the solution
  beyond $T^{\ast}$. This contradicts the assumption $\partial M_{T^{\ast}}$
  is not an equator.
  
  Hence we have shown that $\partial M_{T^{\ast}}$ is an equator. It is easy
  to see that $M_{T^{\ast}}$ must be mean convex with a free boundary, by
  Lemma \ref{mean curvature rigidity}, we have that $M_{T^{\ast}}$ must be a
  totally geodesic disk.
\end{proof}

\section{Willmore type inequality}\label{section willmore}

We prove a version of Theorem \ref{willmore} under the assumption that $M$ is
strictly convex which is a hyperbolic analog of
{\cite{lambert-geometric-2017}}. We basically follow their presentation
adapting to the hyperbolic case and with a slight difference on proving the
equality case of \eqref{willmore for weak convex}. See Lemma \ref{mean
curvature rigidity}. First, we control the $L^2$ norm of $H$.

\begin{lemma}
  Let $M_t$ be the convex solution to \eqref{imcf}. Then for all $1 \leqslant
  p < \infty$, there holds
  \[ \lim_{t \to T^{\ast}} \int_{M_t} H^p (\cdot, t) = 0. \]
\end{lemma}

\begin{proof}
  Since
  \[ - H \langle D z^1, \nu \rangle = \Delta z^1 - n z^1, \]
  and Lemma \ref{estimate normal}, we have
  \begin{equation}
    0 \leqslant \int_{M_t} H (\cdot, t) \leqslant \tfrac{1}{c_0} \int_{M_t}
    \Delta z^1 - n \int_{M_t} z^1 \to 0
  \end{equation}
  as $t \to T^{\ast}$ where the convergence follows from the fact that $M_t$
  converges to a totally geodesic disk. The result then follows from the
  boundness of $H$ in \eqref{curvature bound} and interpolation.
\end{proof}

\begin{theorem}
  \label{willmore for strictly convex}We use the notations in Theorem
  \ref{willmore}. Any strictly convex free boundary hypersurface $M$ in $B_0$
  satisfies the Willmore inequality
  \[ |M|^{\tfrac{2 - n}{n}} \int_M (H^2 - n^2) + \Lambda | \partial M| > - n^2
     \lambda^{\tfrac{2}{n}} + \Lambda \omega_{n - 1} \sinh^{n - 1} \rho_0 . \]
\end{theorem}

\begin{proof}
  We start by calculating the change rate of $\int H^2 - n^2$, we have
\begin{align}
& \partial_t \int_M (H^2 - n^2) \mathrm{d} v \\
= & \int_M (H^2 - n^2) \mathrm{d} v + \int_M 2 H (- \Delta \tfrac{1}{H} -
(- n + |A|^2) \tfrac{1}{H}) \mathrm{d} v \\
= & \int_M (H^2 - n^2) \mathrm{d} v - 2 \int_M H \langle \nabla
\tfrac{1}{H}, \eta \rangle \mathrm{d} v + 2 \int_M \langle \nabla H,
\nabla \tfrac{1}{H} \rangle \mathrm{d} v \\
& - 2 \int_M (|A|^2 - n) \mathrm{d} v \\
\leqslant & \int_M (H^2 - n^2) \mathrm{d} v - 2 \coth \rho_0 | \partial M|
- 2 \int_M \left( \tfrac{H^2}{n} - n \right) \mathrm{d} v \\
= & \tfrac{n - 2}{n} \int_M (H^2 - n^2) \mathrm{d} v - 2 \coth \rho_0 |
\partial M| .
\end{align}
  Now
  \[ \partial_t | \partial M| = \frac{H_{\partial M, \partial B_0}}{H} |
     \partial M| < | \partial M|, \]
  where the inequality is due to convexity of $M$. The volume upper bound
  \[ |M| < \lambda = \omega_{n - 1} \int_0^{\rho_0} \sinh^{n - 1} \rho
     \mathrm{d} \rho . \]
  Let $q (t) = |M|^{\tfrac{2 - n}{n}} \int_M (H^2 - n^2) + \Lambda | \partial
  M|$, so
  \begin{equation}
    q' (t) < \Lambda | \partial M| - 2 \coth \rho_0 |M|^{\tfrac{2 - n}{n}} |
    \partial M| . \label{evolution q}
  \end{equation}
  Recall that $\Lambda = 2 \coth \rho_0 \lambda^{\tfrac{2 - n}{n}}$, with this
  $\Lambda$, we conclude that $q' (t) < 0$. So $q (t)$ is decreasing and
  \[ q (0) > q (T^{\ast}) . \]
  Now we calculate the number
  \[ q (T^{\ast}) = - n^2 |M|^{\tfrac{2}{n}} + \Lambda | \partial M| = - n^2
     \lambda^{\tfrac{2}{n}} + \omega_{n - 1} \Lambda \sinh^{n - 1} \rho_0 . \]
  This concludes our proof.
\end{proof}

We need the exact existence time for the inverse mean curvature flow.

\begin{lemma}
  Suppose that $M_0$ is strictly convex. Then the maximal existence time
  $T^{\ast}$ is
  \[ T^{\ast} = \log \left( \tfrac{\lambda}{|M_0 |} \right), \]
  where $\lambda$ is given in Theorem \ref{convergence theorem}.
\end{lemma}

\begin{proof}
  Using that the evolution of $\tfrac{\partial}{\partial t} g_{i j} = 2 H^{-
  1} h_{i j}$ under inverse mean curvature flow, we see that
  \[ \tfrac{\mathrm{d}}{\mathrm{d} t} |M_t | = |M_t |, \]
  so
  \[ |M_t | = \mathrm{e}^t |M_0 | . \]
  Since we know from Theorem \ref{convergence theorem} that the flow
  convergences to a totally geodesic disk in $C^{1, \alpha}$, we know that
  \[ \lambda = \mathrm{e}^{T^{\ast}} |M_0 |, \]
  and the maximal existence time follows.
\end{proof}

We are using the mean curvature flow with free boundary to approximate a
weakly convex free boundary hypersurface.

\begin{lemma}
  \label{approximation of weak convex}Suppose that $F : \mathbb{D} \times [0,
  T) \to \mathbb{R}^{n + 1}$ is a solution to the mean curvature flow
  \begin{equation}
    \left\{\begin{array}{lll}
      \tfrac{\partial}{\partial t} F & = - H \nu & \text{ in } \mathbb{D}\\
      \langle \nu, \eta \rangle & = 0 & \text{ on } \partial \mathbb{D},
    \end{array}\right. \label{mcf}
  \end{equation}
  with initial hypersurface $M_0$ being weakly convex and perpendicular to the
  sphere from the inside. Then either $\partial M_0$ is an equator of the
  sphere or $h_{i j}$ is positive definite for $t > 0$.
\end{lemma}

\begin{proof}
  Under the mean curvature flow,
  \[ \partial_t h_{i j} = \nabla_i \nabla_j H - H h_i^k h_{j k} - H g_{i j},
  \]
  Using Simons identity \eqref{simons} for $\Delta h_{i j}$, we have the
  evolution of the second fundamental form,
  \[ (\partial_t - \Delta) h_{i j} = n h_{i j} + |A|^2 h_{i j} - 2 H h_i^k
     h_{j k} - 2 H g_{i j} . \]
  If $\partial M_0$ is not an equator, then there exists a strictly convex
  point due to the same reasoning as {\cite[Lemma
  3.1]{lambert-geometric-2017}}. Let
  \[ \chi (\xi, t) = \min_{|v| = 1} h_{i j} v^i v^j . \]
  Since $h_{i j}$ is smooth, $\chi$ is Lipschitz in space and by a simple
  cut-off argument, we can find a smooth function $\phi_0 : M \to \mathbb{R}$
  so that $0 \leqslant \phi_0 \leqslant \chi (\xi, 0)$ and there exists $\xi'
  \in M$ so that $\phi_0 (\xi') > 0$. We extend $\phi_0$ to $\phi :
  \mathbb{D}^n \times [0, \tau) \to \mathbb{R}$ by a heat flow,
  \begin{equation}
    \left\{\begin{array}{rll}
      \tfrac{\partial}{\partial t} \phi - \Delta_{g (t)} \phi & = 0 & \text{
      in } \ensuremath{\operatorname{int}} (\mathbb{D}) \times [0, \tau)\\
      \nabla_{\eta} \phi & = 0 & \text{ on } \partial \mathbb{D} \times [0,
      \tau)\\
      \phi (\cdot, 0) & = \phi_0, & 
    \end{array}\right.
  \end{equation}
  where $\Delta$ is the Laplace-Beltrami operator of the metrics $g (t)$
  induced by the solution $F$ of \eqref{mcf}. This is only a linear parabolic
  PDE so we can find a short time solution on $[0, \tau)$ for some small $\tau
  > 0$ by standard theory. By the strong maximum principle,
  \[ \phi (\xi, t) > 0 \text{ for } \xi \in \mathbb{D}, t > 0. \]
  Let $M_{i j} = h_{i j} - \phi g_{i j}$, then
  \begin{equation}
    (\partial_t - \Delta) M_{i j} = n h_{i j} + |A|^2 h_{i j} - 2 H h_i^k h_{j
    k} - 2 H g_{i j} + 2 \phi H h_{i j} =: N_{i j} . \label{evolution M ij}
  \end{equation}
  Let $v$ be a unit vector such that
  \[ M_{i j} v^j = h_{i j} v^j - \phi g_{i j} v^j = 0. \]
  Then tensor $N_{i j}$ applies to $v \otimes v$ is
\begin{align}
& N_{i j} v^i v^j \\
= & n \phi + |A|^2 \phi - 2 H \phi^2 - 2 H \phi + 2 H \phi^2 \\
= & n \phi + |A|^2 \phi - 2 H \phi \geqslant 0,
\end{align}
  where the inequality $n + |A|^2 \geqslant 2 H$ is simply proved by
  diagonalizing the matrix $h_{i j}$. We have verified that evolution of $M_{i
  j}$ satisfies the null eigenvector condition.
  
  Now we wish to apply Stahl
  {\cite{stahl-convergence-1996,stahl-regularity-1996}} to conclude that $M_{i
  j} \geqslant 0$ on $(0, \delta)$. Because $\phi > 0$ for $t > 0$, $h_{i j} >
  0$ for $0 < t < \tau$ and we finish our proof by applying {\cite[Proposition
  4.5]{stahl-convergence-1996}} to the flow $F (x, t - \tfrac{\tau}{2})$. The
  rest of the proof is devoted to the assertion that $M_{i j} \geqslant 0$.
  
  We use Stahl's notation for comparability, for $p \in \partial B_0$, write
  $\mu \in T_p M$ for the outward pointing normal to $\partial B_0$ and some
  basis tangent vectors $\partial_I$ of $\partial M$ so that $(\mu,
  \partial_I)$ with $2 \leqslant I \leqslant n$ induces some coordinates near
  $p$. Now we show that the conditions in {\cite[Lemma
  3.4]{stahl-regularity-1996}} hold.
  
  Observe that Lemma \ref{lemma normal derivative h ij} holds for any
  hypersurface perpendicular to the geodesic sphere $\partial B_0$, this
  implies that
  \begin{equation}
    \nabla_{\mu} M_{i j} = \coth \rho_0 (h_{\mu \mu} \delta_{I J} - h_{I J}) .
    \label{normal derivative M ij}
  \end{equation}
  Replacing the inverse mean curvature flow of Lemma \ref{lemma normal
  derivative h 11} by mean curvature flow, we have that $\nabla_{\mu} h_{\mu
  \mu} = \coth \rho_0 (2 H - n h_{\mu \mu})$. See also {\cite[Theorem
  4.3]{stahl-convergence-1996}}. So
  \[ \nabla_{\mu} M_{\mu \mu} = \coth \rho_0 (2 H - n h_{\mu \mu}) . \]
  Suppose first that $V \in T_p \partial M$ is a minimal eigenvector with
  eigenvalue $\lambda \in (- \delta, 0]$, that is
  \[ M_{i j} V^i = \lambda g_{i j} V^i . \]
  We see also that $V$ is also a minimal eigenvector of $h_{i j}$, and
  therefore $h_{i j} V^i V^j \leqslant h_{\mu \mu}$. So \eqref{normal
  derivative M ij} implies that
  \[ \nabla_{\mu} M_{I J} V^I V^J \geqslant 0. \]
  Now suppose that $\mu$ is a minimal eigenvector with eigenvalue $\lambda \in
  (- \delta, 0]$. Again minimality of $\mu$ implies that for all $W \in T_p
  \partial M$, $h_{i j} W^i W^j \geqslant h_{\mu \mu}$, in particular, $H
  \geqslant n h_{\mu \mu}$ and so
  \begin{equation}
    \nabla_{\mu} M_{\mu \mu} \geqslant H \geqslant 0.
  \end{equation}
  Now we apply the tensor maximum principle (Theorem 3.3 and Lemma 3.4 of
  {\cite{stahl-regularity-1996}}) to conclude that $M_{i j} \geqslant 0$.
\end{proof}

\begin{corollary}
  \label{approximation by strictly convex hypersurface}Suppose that $M$ is
  weakly convex free boundary hypersurface and $\partial M$ is not an equator.
  Then there exists an $\varepsilon > 0$ such for $0 < t < \varepsilon$ there
  are smooth and strictly convex free boundary hypersurfaces that satisfy
  \begin{equation}
    \int_{M_t} H^2 \to \int_M H^2, |M_t | \to |M|, | \partial M_t | \to |
    \partial M|
  \end{equation}
  as $t \to 0$.
\end{corollary}

\begin{proof}
  It follows from {\cite[Theorem 2.1]{stahl-regularity-1996}} and regularity
  at initial time $t = 0$ of \eqref{mcf}.
\end{proof}

We now prove Theorem \ref{willmore}.

\begin{proof}[Proof of Theorem \ref{willmore}]
  If $\partial M$ is an equator, then $M$ has to be totally geodesic by Lemma
  \ref{mean curvature rigidity} and hence the Willmore inequality
  \eqref{willmore for weak convex} is valid for $M$ with equality.
  
  Due to Theorem \ref{willmore for strictly convex} and \ref{approximation of
  weak convex}, the Willmore inequality \eqref{willmore for weak convex} holds
  for weakly convex hypersurface $M$ with non-equatorial $\partial M$.
  
  We characterize the equality case of \eqref{willmore for weak convex}. If
  $\partial M$ is an equator, by previous argument, $M$ has to be totally
  geodesic. It follows the same lines from {\cite[Lemma
  3.4]{lambert-geometric-2017}} that
  \begin{equation}
    |M| \leqslant |C_M | < \lambda \label{volume estimate}
  \end{equation}
  for some constant $C_M > 0$ if we assume $\partial M$ is not an equator.
  
  Due to Lemma \ref{approximation of weak convex}, for every $\varepsilon > 0$
  there exists a strictly convex hypersurface perpendicular to the sphere
  $\partial B_0$ from the inside $M^{\varepsilon}$ such that
  \[ q (M^{\varepsilon}) \leqslant q (M) + \varepsilon . \]
  We know that the maximal existence time
  \[ T_{\varepsilon}^{\ast} = \log (\tfrac{\lambda}{|M^{\varepsilon} |}) . \]
  By \eqref{evolution q} , the quantity $q^{\varepsilon} (t)$ satisfy$q' (t) <
  \Lambda | \partial M| - 2 \coth \rho_0 |M|^{\tfrac{2 - n}{n}} | \partial M|
  .$
\begin{align}
\tfrac{\mathrm{d}}{\mathrm{d} t} q^{\varepsilon} (t) < & | \partial
M_t^{\varepsilon} | (\Lambda - 2 \coth \rho_0 |M^{\varepsilon}_t
|^{\tfrac{2 - n}{n}}) \\
= & 2 \coth \rho_0 \lambda^{\tfrac{2 - n}{n}} | \partial M_t^{\varepsilon}
| \left( 1 - \mathrm{e}^{\tfrac{n - 2}{n} (T^{\ast}_{\varepsilon} - t)}
\right) .
\end{align}
  Due to \eqref{volume estimate} and Corollary \ref{approximation by strictly
  convex hypersurface}, there exists a positive time $T$ which only depends on
  $|M|$ and is independent of $\varepsilon$ such that
  \begin{equation}
    T_{\varepsilon}^{\ast} \geqslant 2 T > 0.
  \end{equation}
  Hence for all $\varepsilon$ and all $0 \leqslant t \leqslant T$ there holds
  \begin{equation}
    \tfrac{\mathrm{d}}{\mathrm{d} t} q^{\varepsilon} (t) \leqslant c (1 -
    \mathrm{e}^{\tfrac{n - 2}{n} T}) =: - c,
  \end{equation}
  where $c$ only depends on $n$, $|M|$ and $| \partial M|$. Using the strict
  convexity of $M_{\varepsilon}$ and Lemma \ref{willmore for strictly convex}
  , we obtain that
\begin{align}
q (P) < q^{\varepsilon} (T) = & q (M^{\varepsilon}) + \int_0^T
\tfrac{\mathrm{d}}{\mathrm{d} t} q^{\varepsilon} (s) \mathrm{d} s
\\
\leqslant & q (M) + \varepsilon - c T \\
= & q (P) + \varepsilon - c T.
\end{align}
  Choosing a sufficiently small $\varepsilon$, we obtain a contradiction and
  complete the proof.
  
  \ 
\end{proof}

\

\

\

\

\

\

\

\end{document}